\documentclass[12pt]{article}
\usepackage[usenames,dvipsnames]{color}

\usepackage{graphicx}
\usepackage{amsthm}
\usepackage{graphics}
\usepackage{amsmath,,amssymb}
\usepackage{algorithmic}
\usepackage{algorithm}
\usepackage{caption}
\usepackage{subcaption}
\usepackage{pstricks, pst-node}
\usepackage{tikz}
\usepackage{multirow}
\usepackage[round,colon,authoryear]{natbib}

\newtheorem{theorem}{Theorem}
\newtheorem{proposition}{Proposition}
\newtheorem{lemma}{Lemma}

\def\EE{\mathbb{E}}
\def\PP{\mathbb{P}}
\def\RR{\mathbb{R}}
\def\Rcal{\mathcal R}

\def\II{{\mathbb I}}

\def\Ecal{\mathcal E}

\def\Ccal{\mathcal C}
\def\argmin{\mathop{\rm argmin}}

\usepackage{setspace}
\usepackage{footmisc}
\usepackage{lipsum}

\begin{document}
\title{Structured Correlation Detection with Application to Colocalization Analysis in Dual-Channel Fluorescence Microscopic Imaging}

\date{(\today)}
\vskip -10pt

\author{Shulei Wang$^{\ast}$, Jianqing Fan$^\dagger$, Ginger Pocock$^{\mathsection}$ and Ming Yuan$^{\ast,\ddagger}$\\
\\
Morgridge Institute for Research$^{\ast}$ and University of Wisconsin-Madison$^{\ast,\mathsection}$\\
and\\
 Princeton University$^\dagger$}

\footnotetext[1]{
Supported in part by NSF FRG Grant DMS-1265202 and NIH Grant 1U54AI117924-01.}
\footnotetext[2]{
Supported in part by NSF Grants DMS-1206464 and DMS-1406266 and NIH grants R01-GM072611-11.}
\footnotetext[3]{
Ming Yuan wishes to thank Paul Ahlquist, Kevin Eliceiri and Nathan Sherer for introducing him to colocalization analysis in microscopic imaging, and Richard Samworth for helpful discussions and careful reading of an earlier draft that has led to much improved presentation. Address for correspondence: Department of Statistics, University of Wisconsin-Madison, 1300 University Avenue, Madison, WI 53706.}


\maketitle

\begin{abstract}
Motivated by the problem of colocalization analysis in fluorescence microscopic imaging, we study in this paper structured detection of correlated regions between two random processes observed on a common domain. We argue that although intuitive, direct use of the maximum log-likelihood statistic suffers from potential bias and substantially reduced power, and introduce a simple size-based normalization to overcome this problem. We show that scanning with the proposed size-corrected likelihood ratio statistics leads to optimal correlation detection over a large collection of structured correlation detection problems.
\end{abstract}



\newpage
\section{Introduction}
\label{sec:intro}

Most if not all biological processes are characterized by complex interactions among bio-molecules such as proteins. A common way to decipher such interactions is through multichannel fluorescence microscopic imaging where each molecule is labeled with fluorescence of a unique emission wavelength, and their biological interactions can be identified with correlation between the expression of fluorescent proteins in certain compartments. Although an ad hoc approach, visual inspection of the overlayed image from both channels is arguably the most common way to determine colocalization in multichannel fluorescence microscopy. Potential pitfalls of this na\"ive strategy, however, are also well-documented as merged images are heavily influenced by factors such as bleed-through, cross-talk, and relative intensities between different channels. See, e.g., \cite{bolte2006guided} and \cite{comeau2006guide}.

Since the pioneering work of Manders and his collaborators in early 1990s, quantitative methods have also been introduced to colocalization analysis. See, e.g., \cite{manders1992dynamics} and \cite{manders1993measurement}. These approaches typically proceed by first manually selecting a region where the two molecules are likely to colocalize. The degree of colocalization is then measured through various notions of correlation coefficient, most notably Pearson's correlation coefficient or Menders' correlation coefficient, computed specific to the chosen area. See \cite{manders1993measurement}, \cite{costes2004automatic}, Parmryd, \cite{adler2008replicate}, \cite{herce2013new} among others. Obviously, the performance of these approaches depends critically on the manually-selected region of interest, which not only makes the analysis subjective but also creates a bottleneck for high-throughput microscopic imaging processes. Moreover, even if the region is selected in a principled way, colocalization may not be directly inferred from the value of the correlation coefficient computed on the region because the value of the coefficient itself does not translate into statistical significance. This problem could be alleviated using permutation tests, as suggested by \cite{costes2004automatic}. In doing so, however, one still neglects the fact that the region of interest is selected based upon its plausibility of colocalization, and the resulting p-value may appear significant merely because of our failure to adjust for the selection bias. The present work is motivated by this need for an automated and statistically valid way to detect colocalization.

Colocalization analysis can naturally be formulated as an example of a broad class of problems that we shall refer to as structured correlation detection where we observe multiple collections of random variables on a common domain and want to determine if there is a subset of these variables that are correlated. These types of problem arise naturally in many different fields. For example, in finance, detecting time periods where two common stocks show unusual correlation is essential to the so-called pairs trading strategy \citep[see, e.g.][]{pairs04}. Other potential examples of structured correlation detection problems can also be found in \cite{chen1997testing}, \cite{Robinson08}, \cite{Wied11}, and \cite{Rodionov15}, among many others. To fix ideas, in what follows, we shall focus our discussion in the context of colocalization analysis. More specifically, denote by $\II$ the index set of all pixels in the field of view. In a typical two or three dimensional image, $\II$ could be a lattice of the corresponding dimension. In practice, it is also possible that $\II$ is a certain subset of a lattice. For example, when investigating intercellular activities, $\II$ only includes pixels that correspond to the interior of a cell, or a compartment (e.g. nucleus) of a cell. For each location $i\in \II$, let $X_{i}$ and $Y_i$ be the intensities measured at the two channels respectively, as illustrated in the left panel of Figure \ref{fig:intensity}.

\begin{figure}[ht!]
\begin{center}
\begin{tikzpicture}[scale=.6,every node/.style={minimum size=1cm}]
   \begin{scope}[
           yshift=-83,every node/.append style={
           yslant=0.5,xslant=-1},yslant=0.5,xslant=-1
           ]
       \fill[white,fill opacity=0.9] (0,0) rectangle (5,5);
       \draw[step=4mm, red] (0,0) grid (5,5);
       \draw[red,thick] (0,0) rectangle (5,5);
       \fill[red] (1.25,1.55) rectangle (1.55,1.25); 
       \fill[red] (0.85,1.55) rectangle (1.15,1.25); 
       \fill[red] (0.85,1.15) rectangle (1.15,0.85); 
       \fill[red] (1.25,0.75) rectangle (1.55,0.45); 
   \end{scope}
   \begin{scope}[
           yshift=0,every node/.append style={
           yslant=0.5,xslant=-1},yslant=0.5,xslant=-1
           ]
       \fill[white,fill opacity=0.9] (0,0) rectangle (5,5);
       \draw[step=4mm, green] (0,0) grid (5,5); 
       \draw[green,thick] (0,0) rectangle (5,5);
       \fill[green] (1.25,1.55) rectangle (1.55,1.25);
       \fill[green] (0.85,1.55) rectangle (1.15,1.25);
       \fill[green] (0.85,1.15) rectangle (1.15,0.85);
       \fill[green] (1.25,0.75) rectangle (1.55,0.45);
   \end{scope}
%
%
   \draw[-latex,thick,black](-3,5)node[left]{$X_{i}$}
       to[out=0,in=90] (-.4,1.4);
   \draw[-latex,thick,black](-3,1.1)node[left]{$Y_{i}$}
       to[out=0,in=200] (-.5,-1.7);
%
\end{tikzpicture}
\hskip 20pt
\begin{tikzpicture}[scale=.6,every node/.style={minimum size=1cm}]
   \begin{scope}[
           yshift=-53,every node/.append style={
           yslant=0.5,xslant=-1},yslant=0.5,xslant=-1
           ]
       \fill[white,fill opacity=0.8] (0,0) rectangle (5,5);
       \fill[blue!70!white,fill opacity=0.9] (0.5,0.5) rectangle (2,2.5);
       \draw[step=0.5, red] (0,0) grid (5,5);
       \draw[red,thick] (0,0) rectangle (5,5);
   \end{scope}
   \begin{scope}[
           yshift=0,every node/.append style={
           yslant=0.5,xslant=-1},yslant=0.5,xslant=-1
           ]
       \fill[white,fill opacity=0.8] (0,0) rectangle (5,5);
       \fill[blue!70!white,fill opacity=0.9] (0.5,0.5) rectangle (2,2.5);
       \draw[step=0.5, green] (0,0) grid (5,5); 
       \draw[green,thick] (0,0) rectangle (5,5);
   \end{scope}
   \draw[-latex,thick,black](6,3)node[right]{}
       to[out=180,in=90] (-.4,1.4);
   \draw[-latex,thick,black](6,3)node[right]{$R$}
       to[out=180,in=90] (-.4,-0.4);
\end{tikzpicture}
        \caption{Pixel view of dual channel images.}\label{fig:intensity}
 \end{center}
\end{figure}
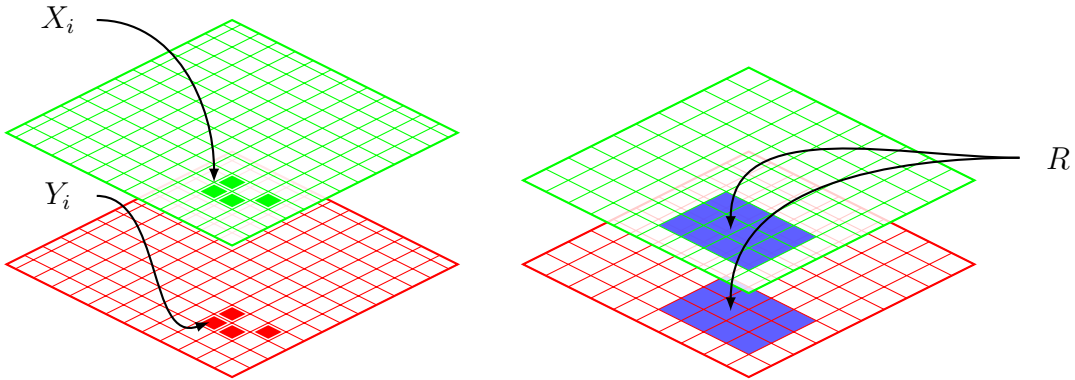

In the absence of colocalization, we assume that $X_{i}$ and $Y_{i}$ are uncorrelated and can be modeled as
\begin{equation}
\label{eq:modelh0}
\begin{bmatrix} X_i\\Y_i \end{bmatrix}\sim N\left(\begin{bmatrix} \mu_1\\\mu_2 \end{bmatrix},\begin{bmatrix} \sigma_1^2&0\\ 0&\sigma_2^2 \end{bmatrix}\right).
\end{equation}
where the marginal means $\mu_1$ and $\mu_2$ and variances $\sigma_1^2$ and $\sigma_2^2$ may be unknown. In the presence of colocalizaton, $X_i$ and $Y_i$ are correlated, and we therefore treat them as observations from a correlated bivariate normal distribution
\begin{equation}
\label{eq:model}
\begin{bmatrix} X_i\\Y_i \end{bmatrix}\sim N\left(\begin{bmatrix} \mu_1\\\mu_2 \end{bmatrix},\begin{bmatrix} \sigma_1^2&\rho\sigma_1\sigma_2\\ \rho\sigma_1\sigma_2&\sigma_2^2 \end{bmatrix}\right).
\end{equation}
When colocalization occurs, it typically does not occur at isolated locations. As a result, the region $R$ of colocalization is more structured than an arbitrary subset of $\II$. For example, colocalization can frequently be observed on a contiguous region $R$, as illustrated in the right panel of Figure \ref{fig:intensity}. Let $\Rcal$ be a library containing all possible regions where correlation may be present. For example, $\Rcal$ could be the collection of all ellipses or polygons on a two dimensional lattice ($\II$).  The primary goal of correlation detection in general, and colocalization analysis in particular, is to check if there is an unknown region $R\in \Rcal$ such that (\ref{eq:modelh0}) holds for all $i\in\II\setminus R$, and (\ref{eq:model}) holds for all $i\in R$ and for some $\rho\neq 0$.

The fact that we do not know on which region $R\in \Rcal$ correlation is present naturally brings about the issue of structured multiple testing. Various aspects of structured multiple testing has been studied in recent years. See, e.g., \cite{desolneux2003maximal}, \cite{perone2004false}, \cite{arias2005near}, \cite{hall2010innovated}, \cite{walther2010optimal}, \cite{arias2011detection}, \cite{FanHanGu12}, \cite{chan2013detection}, and \cite{cai2014rate}, among many others. The problem of colocalization analysis, however, is unique in at least two critical aspects. First, most if not all existing works focus exclusively on signals at the mean level. Our interest here, on the other hand, is on the correlation coefficient. Not only do we want to detect signals in terms of the correlation, but also we want to do so in the presence of unknown marginal means and variances as nuisance parameters. In addition, prior work typically deals with situations where $\II$ is one-dimensional and $\Rcal$ is a collection of segments, which is amenable to statistical analysis and sometimes also allows for fast computation. In the case of colocalization analysis, however, the index set $\II$ is multidimensional and the set $\Rcal$ usually contains more complex geometric shapes. To address both challenges, we develop in this article a general methodology for correlation detection on a general domain that can be readily applied for colocalization analysis.

Our method is motivated by an observation that, for a fairly general family of $\Rcal$, the likelihood ratio statistics exhibit a subtle dependence on the size of a candidate region. As a result, their direct use for correlation detection may lead to nontrivial bias, and substantially reduced power. To overcome this problem, we introduce a size-corrected likelihood ratio statistic and show that scanning with the corrected likelihood ratio statistic yields optimal correlation detection for a large family of $\Rcal$ in the sense that it can detect elevated correlation at a level no other detectors could improve significantly. We show that the corrected likelihood ratio statistic based scan can also be computed efficiently for a large collection of geometric shapes in arbitrary dimension, characterized by their covering numbers under a suitable semimetric. This includes among others, convex polygons or ellipses, arguably two of the most commonly encountered shapes in practice.

The rest of the paper is organized as follows. In the next section, we introduce a size-corrected likelihood ratio statistic for a general index set $\II$ and collection $\Rcal$, and discuss how it can be used to detect colocalization. We shall also investigate efficient implementation as well as theoretical properties of the proposed method. Section \ref{sec:ex} gives several concrete examples of $\II$ and $\Rcal$ and show how the general methodology can be applied to these specific situations. Numerical experiments are presented in Section \ref{sec:num} to further illustrate the merits of the proposed methods. Proofs are given in Section \ref{sec:proof}, with auxiliary results relegated to the Appendix.

\section{Structured Correlation Detection}
\label{sec:meth}

In a general correlation detection problem, $\II$ can be an arbitrary index set and $\Rcal\subset 2^{\II}$ is a given collection of subsets of $\II$, and we are interested in testing the null hypothesis $H_0$ that (\ref{eq:modelh0}) holds for all $i\in \II$ against a composite alternative $H_a$ that (\ref{eq:model}) holds for all $i\in R$ whereas (\ref{eq:modelh0}) holds for all $i\notin R$, for some $R\in \Rcal$. We shall argue in this section that the usual maximum log-likelihood ratio statistic may not be suitable for correlation detection, and introduce a size-based correction to address the problem.

\subsection{Likelihood ratio statistics}
A natural test statistic for our purpose is the scan, or maximum log-likelihood ratio statistic:
$$
L^\ast=\max_{R\in \Rcal} L_R,
$$
where $L_R$ is the log-likelihood ratio statistic for testing $H_0$:
\begin{equation}
\label{eq:lr}
L_{R}=-(|R|-2)\log{(1-r_{R}^{2})}.
\end{equation}
Here $|R|$ is the cardinality of $R$ and $r_{R}$ is Pearson correlation within $R$:
$$r_{R}=\frac{\sum_{i\in R}(X_{i}-\bar{X}_R )(Y_{i}-\bar{Y}_R )}{\sqrt{\sum_{i\in R}(X_{i}-\bar{X}_R )^{2}\sum_{i\in R}(Y_{i}-\bar{Y}_R )^{2}}}$$
where
$$\bar{X}_R ={1\over |R|}\sum_{i\in R} X_i,\qquad{\rm and}\qquad \bar{Y}_R ={1\over |R|}\sum_{i\in R} Y_i.$$
It is worth noting that strictly speaking, $L_R$ defined by (\ref{eq:lr}) is not the genuine likelihood ratio statistic, which would replace the factor $|R|-2$ on the right hand side of (\ref{eq:lr}) by $|R|$. Our modification accounts for the correct degrees of freedom so that, for a fixed uncorrelated region $R$,
$$L_R\approx (|R|-2) {r_R^2\over 1-r_R^2}\sim t^2_{|R|-2}.$$
See, e.g., \cite{Muir08}. Obviously, when $|R|$ is large, $L_R$ approximately follows a $\chi^2_1$ distribution and the effect of such correction becomes negligible.

The use of scan or maximum log-likelihood ratio statistics for detecting spatial clusters or signals is very common across a multitude of fields. See, e.g., \cite{Fan96}, \cite{FanZhang01} and \cite{glaz2001scan} and references therein. Their popularity is also justified as it is well known that scan statistics are minimax optimal if $|R|$ is small when compared with $|\II|$. See, e.g., \cite{lepski2000asymptotically}, \cite{dumbgen2001multiscale}, and \cite{dumbgen2008multiscale}. But we show here that such a strategy may not be effective for correlation detection unless $|R|$ is very small. In particular, we show that, in the absence of a correlated region, the magnitude of $L_R$ depends critically on its size $|R|$, and therefore, the maximum of $L_R$'s over regions of different sizes is typically dominated by those evaluated on smaller regions. As a result, direct use of $L^\ast$ for correlation detection could be substantially conservative in detecting larger correlated regions.

We now examine the behavior of the maximum of $L_R$ for $R\in \Rcal$ of a particular size. Note that it is possible that there is no element in $\Rcal$ that is of a particular size. To avoid lengthy discussion to account for such trivial situations, we shall consider instead the subset
$$
\Rcal(A)=\{R\in \Rcal: |R|\in (A/2,A]\},
$$
for some positive $A$. In other words, $\Rcal(A)$ is the collection of all possible correlated regions of size between $A/2$ and $A$. The factor of $1/2$ is chosen arbitrarily and can be replaced by any constant in $(0,1)$. Basically $\Rcal(A)$ includes elements of $\Rcal$ that, roughly speaking, are of size $A$. It is clear that
$$
L^\ast=\max_A \left\{\max_{R\in \Rcal(A)}L_R\right\}.
$$
We shall argue that $\max_{R\in \Rcal(A)}L_R$ may have different magnitudes for different $A$s under the null hypothesis. In particular, we shall show that for a large collection of $\Rcal(A)$, $\max_{R\in \Rcal(A)}L_R$ can be characterized precisely.

Obviously, the behavior of $\max_{R\in \Rcal(A)}L_R$ depends on the complexity of $\Rcal(A)$. More specifically, we shall first assume that the possible correlated regions are indeed more structured than arbitrary subsets of $\II$ in that there exist constants $c_1,c_2>0$ independent of $A$ and $n:=|\II|$ such that
\begin{equation}
\label{eq:comp}
|\Rcal(A)|\le c_1 nA^{c_2}.
\end{equation}
In other words, (\ref{eq:comp}) dictates that $|\Rcal(A)|$ increases with $A$ only polynomially. This is to be contrasted with the completely unstructured setting where $\Rcal=2^{\II}$, the collection of all subsets of $\II$, and the number of all subsets of $\II$ of size $A$ is of the order $n^A$, which depends on $A$ exponentially. Condition (\ref{eq:comp}) essentially requires that $\Rcal$ is a much smaller subset of $2^\II$ and therefore indeed imposes structures on the possible regions of correlation.

Na\"ive counting of the size of $\Rcal(A)$ as above, however, may not reflect its real complexity. To this end, we also need to characterize the dissimilarity of elements of $\Rcal(A)$. For any two sets $R_1,R_2\in 2^{\II}$, write
$$
d(R_1,R_2)=1-{|R_1\cap R_2|\over \sqrt{|R_1||R_2|}}.
$$
It is easy to see that $d(\cdot,\cdot)$ is a semimetric on $2^{\II}$. We now consider the covering number of sets of a particular size in $\Rcal$ under $d$. Let $N(A,\epsilon)$ be the smallest integer such that there is a subset, denoted by $\Rcal_{\rm app}(A,\epsilon)$, of $\Rcal$ with
$$|\Rcal_{\rm app}(A,\epsilon)|=N(A,\epsilon)$$
and
$$
\sup_{R_1\in \Rcal(A)}\inf_{R_2\in \Rcal_{\rm app}(A,\epsilon)} d(R_1,R_2)\le \epsilon.
$$
It is worth emphasizing that we require the covering set $\Rcal_{\rm app}(A,\epsilon)\subset\Rcal$. It is clear that $N(A,\epsilon)$ is a decreasing function of $\epsilon$ and $N(A,0)=|\Rcal(A)|$. We shall also adopt the convention that $N(A,1)$ represents the largest number of non-overlapping elements from $\Rcal(A)$. Clearly, without any structural assumption, we can always divide $\II$ into $n/A$ subsets of size $A$. We shall assume that the collection $\Rcal(A)$ is actually rich enough that
\begin{equation}
\label{eq:entropy0}
N(A,1)\ge c_3{n\over A},
\end{equation}
for some constant $c_3>0$. Conversely, we shall assume also that there are not too many ``distinct'' sets in $\Rcal(A)$ in that there are certain constants $c_4,c_5,c_6>0$ independent of $A$ and $N$ such that
\begin{equation}
\label{eq:entropy}
N(A,\epsilon)\le c_4{n\over A}\left(\log {n\over A}\right)^{c_5}\left(1\over \epsilon\right)^{c_6}.
\end{equation}

Conditions (\ref{eq:comp}), (\ref{eq:entropy0}) and (\ref{eq:entropy}) are fairly general and hold for many common choices of $\Rcal$. Consider, for example, the case when $\II=\{1,2,\ldots, n\}$ is a one-dimensional sequence and
$$\Rcal=\{(a,b]: 0\le a<b\le n\}$$
is the collection of all possible segments on $\II$. It is clear that there are at most $n-\ell$ segments of length $\ell$ for any $\ell\in (A/2,A]$, which means
$$
|\Rcal(A)|\le {1\over 2}nA.
$$
In addition, for any $A$, there are at least $\lfloor n/A\rfloor$ distinct segments
$$\{((i-1)A, iA]: i=1,\ldots, \lfloor n/A\rfloor\},$$
of length $A$, implying that (\ref{eq:entropy0}) also holds. On the other hand, it is not hard to see that the collection of all segments starting at $(i-1)\epsilon A/2$ ($i=1,2,\ldots$) of length between $A/2$ and $A$ can approximate any segment of length between $A/2$ and $A$ with approximation error $\epsilon$. Therefore,
$$
N(A,\epsilon)\le \left(A/2\over \epsilon A/2\right) \left(n\over \epsilon A/2\right)=2{n\over A}\left(1\over \epsilon\right)^2,
$$
so that (\ref{eq:entropy}) also holds. In the next section, we shall consider more complex examples motivated by colocalization analysis and show that these conditions are expected to hold in fairly general settings.

We now show that if $\Rcal(A)$ satisfies these conditions, $\max_{R\in \Rcal(A)}L_R$ concentrates sharply around $2\log(n/A)$.

\begin{theorem}
\label{th:emp}
Suppose that (\ref{eq:modelh0}) holds for all $i\in \II$. Assume also that (\ref{eq:comp}) and (\ref{eq:entropy}) hold. Then
\begin{equation}
\label{eq:maxupper}
\max_{R\in \Rcal(A)}L_R\le  2\log\left({n/A}\right)+O_p\left(\log\log\left({n/A}\right)\right),\qquad {\rm as\ }n\to \infty.
\end{equation}
If in addition, (\ref{eq:entropy0}) holds, then
\begin{equation}
\label{eq:maxeq}
\max_{R\in \Rcal(A)}L_R=2\log\left({n/A}\right)+O_p\left(\log\log\left({n/A}\right)\right),\qquad {\rm as\ }n\to \infty.
\end{equation}
\end{theorem}

We adopted a generic chaining \citep[see, e.g.,][]{Tal00} argument for the proof of Theorem \ref{th:emp}. The analysis itself may be of independent interest and can be applied to other similar problems such as deriving asymptotic bounds for likelihood ratio statistics in structured detection of mean shifts.

\subsection{Size-corrected likelihood ratio statistics}
An immediate consequence of Theorem \ref{th:emp} is that the value of $L^\ast$ alone may not be a good measure of the evidence of correlation. It also depends critically on the size of $R$ for which $L_R$ is maximized. As such, when using $L^\ast$ as a test statistic, the critical value is largely driven by $\max_{R\in \Rcal(A)}L_R$ corresponding to smaller $A$'s. Therefore, a test based on $L^\ast$ could be too conservative when correlation is present on a region with a large cardinality. Motivated by this observation, we now consider normalizing $\max_{R\in \Rcal(A)}L_R$, leading to a size-corrected log-likelihood ratio statistic:
\begin{eqnarray*}
T^\ast&=&\max_A \left\{{1\over \log \log(n/A)}\left[\max_{R\in \Rcal:|R|=A}L_R- 2 \log \left(n/A\right)\right]\right\}\\
&=&\max_{R\in \Rcal} \left\{{1\over \log\log(n/|R|)}\left[L_R-2\log\left({n\over |R|}\right)\right]\right\}.
\end{eqnarray*}
For brevity, we shall hereafter assume that $\max_{R\in \Rcal}|R|\le n/4$. In general, we can always replace $\log x$ by $\log_+(x):=\log(\max\{x,1\})$ to avoid the trivial cases where the logarithms may not be well defined.

It is clear that under the null hypothesis, the distribution of $T^\ast$ is invariant to the nuisance parameters and therefore can be readily evaluated through Monte Carlo simulation. More specifically, one can simulate $(X_{i}^\ast,Y_{i}^\ast)^\top\sim N(0,I_2)$ independently for $i\in \II$, and compute $T^\ast$ for the simulated data. The distribution of $T^\ast$ can be approximated by the empirical distribution of the test statistics estimated by repeating this process. Denote by $q_\alpha$ the $(1-\alpha)$-quantile of $T^\ast$ under the null hypothesis. We shall then proceed to reject $H_0$ if and only if $T^\ast>q_\alpha$. This clearly is an $\alpha$-level test by construction. We shall show that in Section \ref{sec:th}, it is also a powerful test for detecting correlation.

One of the potential challenges for scan statistics is computation. To compute $T^\ast$, we need to enumerate all elements in $\Rcal$, which could be quite burdensome. A key insight obtained from studying $T^\ast$ however suggests an alternative to $T^\ast$ that is more amenable for computation. More specifically, it is noted that although numerous, regions of large size, namely $\Rcal(A)$ with a large $A$, may have fewer ``distinct'' elements. As such, we do not need to evaluate $L_R$ on each $R\in \Rcal(A)$ but rather on a smaller covering set $\Rcal(A)$.

With slight abuse of notation, write
$$
\Rcal_k=\{R\in \Rcal: |R|\in (2^{-k}n,2^{-(k-1)}n]\},\qquad k=2,\ldots,\lfloor\log_2 n\rfloor+1.
$$
It is clear that $T^\ast=\max_k T_k^\ast$ where
$$
T_k^\ast=\max_{R\in \Rcal_k} \left\{{1\over \log\log(n/|R|)}\left[L_R-2\log\left({n\over |R|}\right)\right]\right\}.
$$
It turns out that for
$$
k\le k_\ast:=\lfloor\log_2 n-2\log_2\log n\rfloor,
$$
we can approximate $T_k^\ast$ very well by scanning through only a small number of $R$s from $\Rcal_k$. In particular, let $\tilde{\Rcal}_k$ be a $1/(4k^2)$ covering set of $\Rcal_k$ with
$$
|\tilde{\Rcal}_k|=N\left(2^{-(k-1)}n,{1\over 4k^2}\right).
$$
We shall proceed to approximate $T^\ast_k$ by
$$
\tilde{T}_k^\ast=\max_{R\in \tilde{\Rcal}_k} \left\{{1\over \log\log(n/|R|)}\left[L_R-2\log\left({n\over |R|}\right)\right]\right\},
$$
when $k\le k_\ast$. Denote by
$$
\tilde{T}^\ast=\max_k \tilde{T}_k^\ast,
$$
where, with slight abuse of notation, $\tilde{T}_k^\ast=T_k^\ast$ for $k>k_\ast$. Instead of using $T^\ast$, we shall now consider $\tilde{T}^\ast$ as our test statistic. As before, we can compute the $1-\alpha$ quantile $\tilde{q}_\alpha$ of $\tilde{T}^\ast$ under the null hypothesis by Monte Carlo method and proceed to reject $H_0$ if and only if $\tilde{T}^\ast>\tilde{q}_\alpha$.

Compared with $T^\ast$, the new statistic $\tilde{T}^\ast$ is much more computationally friendly. More specifically, under the complexity condition (\ref{eq:entropy}), it amounts to computing the corrected likelihood ratio statistic on a total of
\begin{eqnarray*}
&&\sum_{k\le k_\ast} N\left(2^{-(k-1)}n,{1\over 4k^2}\right)+\sum_{k>k_\ast} N\left(2^{-(k-1)}n,0\right)\\
&\le& c_4(\log 2)^{c_5}4^{c_6}n(\log n)^{c_5+2c_6+1}+c_{1}n(\log n)^{2c_2+1}
\end{eqnarray*}
sets. In other words, the number of size-corrected likelihood ratio statistics we need to evaluate in computing $\tilde{T}^\ast$ is linear in $n$, up to a certain polynomial of logarithmic factor.

\section{Correlation Detection on a Lattice}
\label{sec:ex}

While a general methodology was presented for correlation detection under a generic domain in the previous section, we now examine more specific examples motivated by colocalization analysis in microscopic imaging, and discuss further the operating characteristics of the proposed approach. In particular, we shall focus on correlation detection in a two-dimensional lattice where $\II=\{(i,j): 1\le i,j\le m\}$ so that $n=m^2$, for concreteness, although the discussion can be extended straightforwardly to more general situations such as rectangular or higher order lattices.

Most of the imaging tools allow users to visually identify areas of colocalization using variants of the Lasso tool. This allows either a convex polygonal or ellipsoidal region to be selected. Motivated by this, we shall consider specifically in this section the detection of correlation on either an unknown convex polygonal or ellipsoidal region on a two-dimensional lattice. We show that in both cases, the collection $\Rcal$ of all possible correlated areas satisfies conditions (\ref{eq:comp}), (\ref{eq:entropy0}) and (\ref{eq:entropy}) and therefore the size-corrected scan statistic $\tilde{T}^\ast$ can be efficiently computed.

\subsection{Polygons}
We first treat convex $k$-polygons. Any $k$-polygon can be indexed by its vertices $\{(a_i,b_i): 1\le i\le k\}$, and will therefore be denoted by $K(\{(a_i,b_i): 1\le i\le k\})$. For expositional ease, we focus on the case when the vertices are located on the lattice, although the general case can also be treated with further care. The convexity of a polygon allows us to define its center as $(\bar{a},\bar{b})$ where
$$
\bar{a}={1\over k}\sum_{i=1}^k a_k,\qquad {\rm and}\qquad \bar{b}={1\over k}\sum_{i=1}^k b_i.
$$
Denote by
$$
r_i=\sqrt{(a_i-\bar{a})^2+(b_i-\bar{b})^2}
$$
the distance from the $i$th vertex to the center. To fix ideas, we will focus attention on nearly regular polygons, where $r_i$s are of the same order. In this case, the collection of possible correlated regions is:
$$
\Rcal_{\rm polygon}(k,M)=\left\{K(\{(a_i,b_i): 1\le i\le k\}): \max_{i}{r_{i}}/\min_{i}{r_{i}}\le M\right\}.
$$
Recall that
$$
\Rcal_{\rm polygon}(A;k,M)=\left\{R\in \Rcal_{\rm polygon}(k,M): |R|\in (A/2,A]\right\}.
$$
The following result states that (\ref{eq:comp}) holds for $\Rcal_{\rm polygon}(k,M)$.

\begin{proposition}
\label{pr:poly1}
There exists a constant $c>0$ depending on $k$ and $M$ only such that
$$
|\Rcal_{\rm polygon}(A;k,M)|\le cn A^k.
$$
\end{proposition}

We now verify (\ref{eq:entropy0}) for $\Rcal_{\rm polygon}(k,M)$. To this end, we note that any convex $k$-polygon can be identified with a minimum bounding circle as shown in Figure \ref{fig:poly}. Clearly if two polygons intersect, so do their minimum bounding circles. This immediately implies that (\ref{eq:entropy0}) holds, because we can always place $\lfloor m/r\rfloor^2$ mutually exclusive circles of radius $r$ over an $m\times m$ lattice.

\begin{figure}[htbp]
\begin{center}
\begin{tikzpicture}
 \begin{scope}
       \fill[white,fill opacity=0.8] (-1,-1) rectangle (5,5);
       \draw[step=0.5, black!30!white] (-1,-1) grid (5,5);

       \draw[line width=1pt,black!70!white,dashed] (2,2) circle (2.5);

       \path[draw,line width=2pt,black!70!white] ({2-sqrt(6)},2.5) -- (1,4)--(4,3.5)--(3.5,0.5)--(0.5,0)--cycle;
       \path[draw,line width=2pt,black!70!white,dashed] (2,2) -- (4,3.5);
       \path[draw,line width=2pt,black!70!white,dashed] (2,2) -- (3.5,0.5);
       \path[draw,line width=2pt,black!70!white,dashed] (2,2) -- (1,4);
       \path[draw,line width=2pt,black!70!white,dashed] (2,2) -- ({2-sqrt(6)},2.5);
       \path[draw,line width=2pt,black!70!white,dashed] (2,2) -- (0.5,0);
 \end{scope}
\end{tikzpicture}
\end{center}
\caption{Convex polygon and its minimum bounding circle.}
\label{fig:poly}
\end{figure}
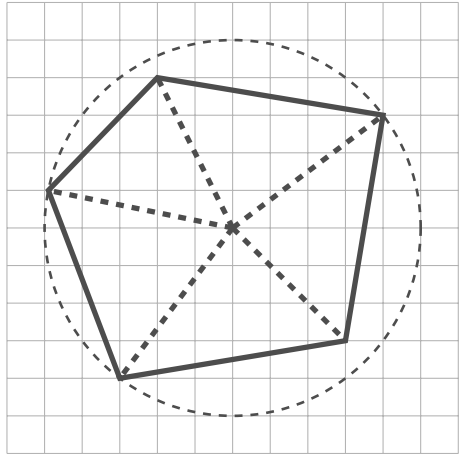

Finally, we show (\ref{eq:entropy}) also holds for $\Rcal_{\rm polygon}(k,M)$ by constructing an explicit covering set. The idea is fairly simple -- we apply a local perturbation to each vertex:
$$
\pi_s(K(\{(a_i,b_i): 1\le i\le k\}))=K(\{(2^{s}\lfloor2^{-s}a_{i}\rfloor,2^{s}\lfloor2^{-s}b_{i}\rfloor): 1\le i\le k\}).
$$
It can be shown that
\begin{proposition}
\label{pr:poly2}
Let $\pi_s$ be defined above. Then there exists an absolute constant $c>0$ such that
$$
\rho(K(\{(a_i,b_i): 1\le i\le k\}),\pi_s(K(\{(a_i,b_i): 1\le i\le k\})))\ge 1-c(\min_i r_i)^{-1}2^{s}.
$$
\end{proposition}

It is clear that, there exist constants $0<c_7<c_8$ depending on $k$ and $M$ only such that
$$
\Rcal_{\rm polygon}(A;k,M)\subset \left\{K\in\Rcal_{\rm polygon}(k,M): c_7A^{1/2}\le r_i\le c_8A^{1/2}, i=1,2,\ldots,k\right\}.
$$
Therefore, by taking $s=\log_2\left(\epsilon A^{1/2}\right)$, we get
$$
N(A,\epsilon)\le c_9\frac{n}{A}\left(\log\left(\frac{n}{A}\right)\right)^{k-1}\left(\frac{1}{\epsilon}\right)^{2k+2}
$$
In addition, this argument suggests a simple strategy by {\it digitalization} ($\pi_s$) to construct a covering set for $\Rcal$.

From this particular case, we can see the tremendous computational benefit of $\tilde{T}^\ast$ over $T^\ast$. To evaluate $T^\ast$, we need to compute the size-corrected likelihood ratio statistics for a total of $|\Rcal|=O(n^k)$ possible regions. In contrast, computing $\tilde{T}^\ast$ only involves $O\left(n{\rm polylog}(n)\right)$ regions as shown in the previous section. Here ${\rm polylog}(\cdot)$ stands for a certain polynomial of $\log(\cdot)$.

\subsection{Ellipses}
Next, we consider the case when $\Rcal$ is a collection of ellipses on a two-dimensional lattice. Recall that any ellipse can be indexed by its center $(\tau_1,\tau_2)^\top$, and a positive definite matrix $\Sigma\in \RR^{2\times 2}$:
$$
\Ecal((\tau_1,\tau_2)^\top, \Sigma)=\left\{(x_1,x_2)^\top\in \RR^2: \left(x_1-\tau_1,x_2-\tau_2\right)\Sigma^{-1}\left(\begin{array}{c}x_1-\tau_1\\ x_2-\tau_2\end{array}\right)\le 1\right\}.
$$
For brevity, we shall consider the case when $\Sigma$ is well conditioned in that its condition number, that is the ratio between its eigenvalues, is bounded to avoid lengthy discussion about the effect of discretization. In this case,
$$
\Rcal_{\rm ellipse}=\left\{\Ecal((\tau_1,\tau_2)^\top, \Sigma)\cap \II: 1\le \tau_1,\tau_2\le m, \Sigma\succ 0, \lambda_{\max}(\Sigma)/\lambda_{\min}(\Sigma)\le M\right\}.
$$
We first note that any ellipse can be identified with its circumscribing rectangle as shown in Figure \ref{fig:eclipse}.
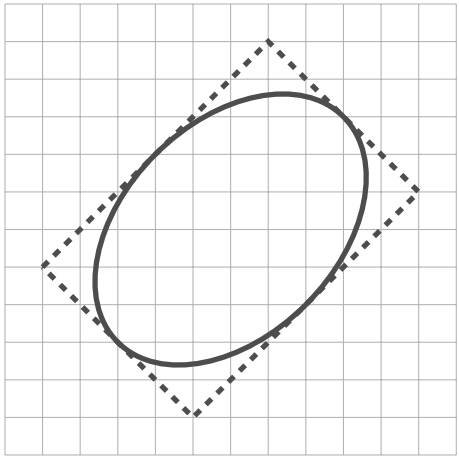
\begin{figure}[htbp]
\begin{center}
\begin{tikzpicture}
 \begin{scope}
       \fill[white,fill opacity=0.8] (-3,-3) rectangle (3,3);
       \draw[step=0.5, black!30!white] (-3,-3) grid (3,3);

       \draw[line width=2pt,black!70!white,dashed,rotate=-45] ({-sqrt(2)},{-1.5*sqrt(2)}) rectangle ({sqrt(2)},{1.5*sqrt(2)});
       \path[draw,line width=2pt,black!70!white,rotate=-45] (0,0) ellipse ({sqrt(2)} and {1.5*sqrt(2)});
 \end{scope}
\end{tikzpicture}
\end{center}
\caption{Circumscribing rectangle of an ellipse}
\label{fig:eclipse}
\end{figure}
Therefore, immediately following the bound on the number of rectangles on a lattice, for example by Proposition \ref{pr:poly1} with $k=4$, we get
$$
\Rcal_{\rm ellipse}\le cnA^4,
$$
for some constant $c>0$. Similarly, if two ellipses intersect, then so do their minimum bounding rectangles. By the argument for polygons, we therefore know that (\ref{eq:entropy0}) and (\ref{eq:entropy}) also hold for $\Rcal_{\rm ellipse}$.

\section{Optimality}
\label{sec:th}

We now study the power of the proposed test $T^\ast$ and its variant $\tilde{T}^\ast$. We shall first investigate the required strength of correlation so it can be detected using the proposed tests.

\begin{theorem}
\label{th:upper}
Assume that (\ref{eq:comp}) and (\ref{eq:entropy}) hold. If there exists a correlated region $R\in \Rcal$, with $|R|\to \infty$, such that (\ref{eq:modelh0}) holds for $i\notin R$ and (\ref{eq:model}) holds for $i\in R$, and
\begin{equation}
\label{eq:upper}
|R|\log\left(1\over 1-\rho^2\right)\ge (2+\delta_n) \log\left({n\over |R|}\right)
\end{equation}
for some $\delta_n>0$ such that $\delta_n\log^{1/2} (n/|R|)\to \infty$, then $T^\ast>q_\alpha$ and $\tilde{T}^\ast>\tilde{q}_\alpha$ with probability tending to one as $n\to\infty$.
\end{theorem}

Theorem \ref{th:upper} shows that whenever correlation on a region $R$ satisfies (\ref{eq:upper}), our tests will consistently reject the null hypothesis and have power tending to one. The detection boundary of the proposed tests for a correlated region $R$ can therefore be characterized by (\ref{eq:upper}). More specifically, depending on the cardinality $|R|$, there are three different regimes. 
\begin{itemize}
\item[$\bullet$] For large regions where $|R|\asymp n$, correlation is detectable if $|R|\rho^2\to \infty$. Recall that, from Neyman-Pearson Lemma, even if the correlated region $R$ is known in advance, we can only consistently detect it under the same requirement. Put differently, the proposed method is as powerful as if we knew the region in advance.
\item[$\bullet$] For regions of intermediate sizes such that $\log n\ll |R|\ll n$, the detection boundary becomes $\rho^2\ge(2+\delta_n)|R|^{-1}\log(n/|R|)$, provided that $\delta_n\sqrt{\log(n/|R|)}\to \infty$. Here, we can see that weaker correlation can be detected over larger regions.
\item[$\bullet$] And finally for small regions where $|R|\ll \log(n)$, detection is only possible for nearly perfect correlation in that $\rho^2\ge 1-\exp(-(2+\delta_n)\log(n)/|R|)$ where $\delta_n\sqrt{\log n}\to \infty$.
\end{itemize}

It turns out that the detection boundary achieved by $T^\ast$ and $\tilde{T}^\ast$ as shown in Theorem \ref{th:upper} is indeed sharply optimal.
\begin{theorem}
\label{th:lower}
Assume that (\ref{eq:entropy0}) holds. For any $\alpha$-level test $\Delta$, there exists an instance where correlation occurs on some $R\in \Rcal$ obeying
\begin{equation}
\label{eq:lower}
|R|\log\left(1\over 1-\rho^2\right)\ge (2-\delta_n) \log\left({n\over |R|}\right)
\end{equation}
for a certain $\delta_n>0$ with $\delta_n\log^{1/2} (n/|R|)\to \infty$, such that the type II error of $\Delta$ converges to $1-\alpha$ as $n\to \infty$. Moreover, if there exists some $\alpha$-level test $\Delta$ for which the type II error converges to 0 as $n\to \infty$ on any instance where
correlation occurs on some $R\in \Rcal$ obeying
\begin{equation}
\label{eq:lower1}
|R|\log\left(1\over 1-\rho^2\right)\ge c_{n}\qquad\mbox{ and }\qquad|R|\rightarrow \infty,
\end{equation}
then it is necessary to have $c_{n}\rightarrow \infty$ as $n\rightarrow \infty$.
\end{theorem}

In other words, Theorem \ref{th:lower} shows that any test is essentially powerless for detecting correlation with
$$
|R|\log\left(1\over 1-\rho^2\right)\le (2-\delta_n) \log\left({n\over |R|}\right)
$$
for any $\delta_n>0$ such that $\delta_n\log^{1/2} (n/|R|)\to \infty$. Together with Theorem \ref{th:upper}, we see that, when $n/|R|\to \infty$, the optimal detection boundary for colocalization for a general index set $\II$ and a large collection of $\Rcal$'s that satisfy certain complexity requirements is
$$
|R|\log\left(1\over 1-\rho^2\right)=2 \log\left({n\over |R|}\right);
$$
and the size-corrected scan statistic is sharply optimal.

The second statement of Theorem \ref{th:lower} deals with the case when $\lim\sup n/|R|$ is finite. Together with Theorem \ref{th:upper}, (\ref{eq:lower1}) implies that in this case, the correlated region can be detected if and only if
$$
\rho^{2}|R|\to \infty
$$
and size-corrected scan statistic is again optimal.

To better appreciate the effect of the size of a correlated region on its detectability, it is instructive to consider the cases where $|R|=n^\alpha$ for some $0<\alpha<1$ or $|R|=(\log n)^\alpha$ for some $\alpha>1$. In the former case when $|R|=n^\alpha$, the detection boundary is
$$
\rho^2= 2(1-\alpha)n^{-\alpha}\log n.
$$
In the latter case when $|R|=(\log n)^\alpha$, the detection boundary is
$$
\rho^2= 2(\log n)^{\alpha-1}.
$$
In both cases, it is clear that much weaker correlation can be detected on larger regions.

\section{Numerical Experiments}
\label{sec:num}
We now conduct numerical experiments to further demonstrate the practical merits of the proposed methodology.

\subsection{Simulation}
We begin with a series of four sets of simulation studies. To fix ideas, we focus on two dimensional lattices in our simulation studies. The first set of simulations was designed to show the flexibility of the general method by considering a variety of different shapes of correlated regions, namely the choice of the library $\Rcal$, including axis-aligned rectangles, triangles and axis-aligned ellipses. We compare the performance of size-corrected likelihood ratio statistic and the uncorrected likelihood ratio statistic to demonstrate the necessity and usefulness of the proposed correction. The second set was carried out to compare the full scan statistic $T^{*}$ and the nearly linear time scan $\tilde{T}^{*}$ and illustrate similar performance between the two methods yet considerable computation gain by using $\tilde{T}^{*}$. The third and fourth sets of simulation studies were conducted to confirm qualitatively our theoretical findings about the effect of the size $|\II|$ of the lattice and the area $A$ of correlated region on its detectability. In each case, we shall assume that only the shape of the correlated region is known and therefore $\Rcal$ is the collection of all regions of a particular shape. In addition, we simulate the null distribution and identify the upper $5\%$ quantile of the null distribution based on $1000$ Monte Carlo simulations. We reject the null hypothesis for a simulation run if the corresponding test statistic, $T^\ast$, $\tilde{T}^\ast$ or $L^\ast$, exceeds their respective upper quantile. This ensures that each test is at level $5\%$, up to Monte Carlo simulation error.

As argued in the previous sections, our methods can handle a variety of geometric shapes. We now demonstrate this versatility through simulation where we consider detecting a correlated region in the form of a triangle, an ellipse or a rectangle. In particular, we simulated data on a $32\times 32$ squared lattice. Correlation was imposed on a right triangle with side length $10$, $20$ and $10\sqrt{5}$, or an axis-aligned ellipse with short axis $4.94$ and long axis $6.36$, or a rectangle of size $10\times 10$. The location of these correlated regions was selected uniformly over the lattice. Typical simulated examples of different correlation and shape are given in Figure \ref{fig:simex}.
\begin{figure}
    \centering
    \begin{subfigure}[b]{0.3\textwidth}
    \centering
        \includegraphics[width=0.8\textwidth]{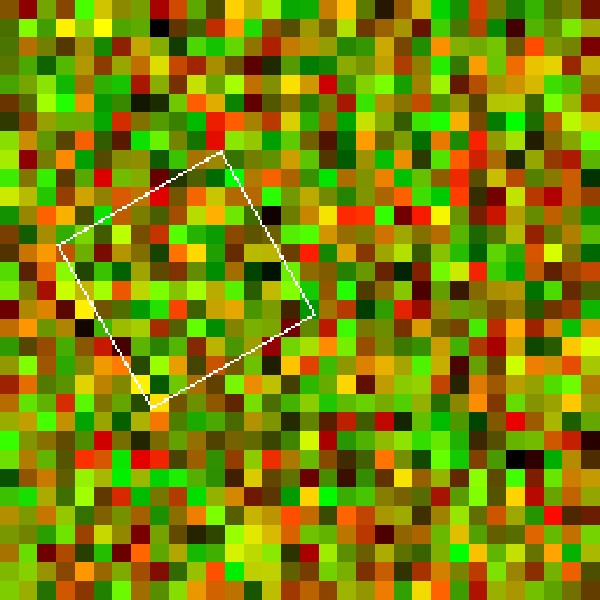}
        \caption{Rectangle with $\rho=0.2$.}
    \end{subfigure}
    \begin{subfigure}[b]{0.3\textwidth}
    \centering
        \includegraphics[width=0.8\textwidth]{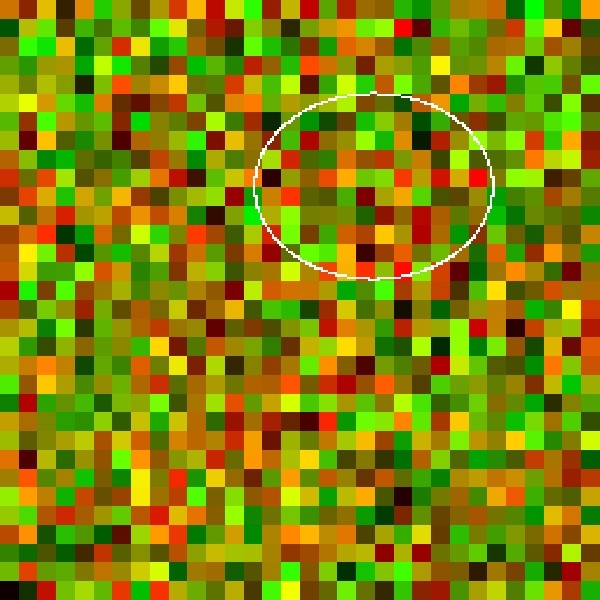}
        \caption{Ellipse with $\rho=0.2$.}
    \end{subfigure}
    \begin{subfigure}[b]{0.3\textwidth}
    \centering
        \includegraphics[width=0.8\textwidth]{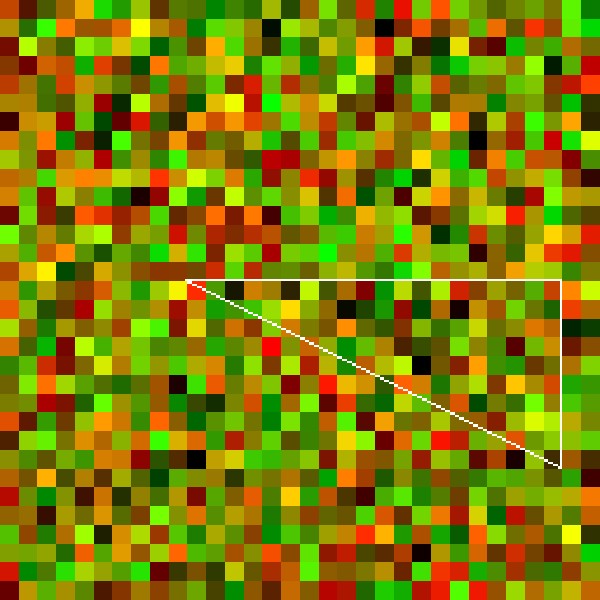}
        \caption{Triangle with $\rho=0.2$.}
    \end{subfigure}
    \caption{Simulated examples: overlayed images from green channel and red channel in a typical simulation run for different shapes of correlated regions.}\label{fig:simex}
\end{figure}

To assess the power of $T^\ast$, we considered two relatively small values of correlation coefficient $\rho$: $0.2$ and $0.4$. For comparison purposes, we computed for each simulation run both $T^\ast$ and the uncorrected maximum likelihood ratio statistic $L^\ast$. The experiment was repeated for $500$ times for each combination of shape and correlation coefficient. The results are summarized in Table \ref{tb:dfshape}. These results not only show the general applicability of our methods but also demonstrate the improved power of the size correction we apply.

\begin{table}[htbp]
\begin{center}
\begin{tabular}{c|cccccc}
\hline
\hline
Shape & \multicolumn{2}{c} {Rectangle} & \multicolumn{2}{c} {Ellipse} & \multicolumn{2}{c} {Triangle}\\
\hline
$\rho$& 0.2 & 0.4 &  0.2 & 0.4 &  0.2 & 0.4 \\
\hline
$T^\ast$ & 0.16 & 0.42& 0.25 & 0.6 & 0.21 & 0.58\\
$L^\ast$ & 0.04 & 0.20& 0.03& 0.51& 0.03& 0.26\\
\hline
\end{tabular}
\end{center}
\caption{Power comparison between $T^\ast$ and $L^\ast$ for different combinations of shape and correlation coefficient.}
\label{tb:dfshape}
\end{table}%

We now compare the full scan statistic $T^\ast$ with its more computationally efficient variant $\tilde{T}^\ast$. We focus on the case when the correlated region is known to be an axis-aligned rectangle. The true correlated region is a randomly selected $10\times 10$ rectangle on a $64\times 64$ lattice. We consider a variety of different correlation coefficients $0.2$, $0.4$, $0.6$, and $0.8$. The performance and computing time of both tests are reported in Table \ref{tb:dfmthd}, which is also based on 500 runs for each value of the correlation coefficient. It is clear from Table \ref{tb:dfmthd} that the two tests enjoy similar performance with $T^\ast$ slightly more powerful. Yet $\tilde{T}^\ast$ is much more efficient to evaluate as expected.

\begin{table}[htbp]
\begin{center}
\begin{tabular}{cc|cccc}
\hline
\hline
&Correlation Coefficient & 0.2 & 0.4 & 0.6 & 0.8\\
\hline
Power & $T^{*}$  & 0.108 & 0.228 & 0.502 & 0.708\\
& $\tilde{T}^{*}$ & 0.106 & 0.214 & 0.410 & 0.606\\
\hline
Time (ms)& $T^{*}$ & 444.084 & 447.236 & 452.634 & 453.064\\
&$\tilde{T}^{*}$ & 139.026 & 139.344 & 140.554 & 142.144\\
\hline
\end{tabular}
\end{center}
\caption{Comparison between $T^{*}$ and $\tilde{T}^{*}$.}
\label{tb:dfmthd}
\end{table}%

We note that the computing gain of $\tilde{T}^\ast$ over $T^\ast$ becomes more significant for larger images. In particular, we ran similar scans over lattices of size $256\times 256$, $256\times 512$ and $512\times 512$. The computing time for a typical dataset is presented in Table \ref{tb:dfmthd1}:

\begin{table}[htbp]
\begin{center}
\begin{tabular}{c|ccc}
\hline
\hline
Size of Lattice & $256\times 256$ & $256\times 512$ & $512\times 512$\\
\hline
\hline
Computing time of $T^{*}$ (s)& 129.942 & 487.238 & 1934.996\\
Computing time of $\tilde{T}^{*}$ (s)& 16.59 &  45.117 & 144.206\\
\hline
\end{tabular}
\end{center}
\caption{Comparison of computing times for $T^{*}$ and $\tilde{T}^{*}$.}
\label{tb:dfmthd1}
\end{table}%

We now evaluate the effect of the size of a correlated region on its detectability. In the light of the observations made in the previous set of experiments, we focus on using $\tilde{T}^\ast$ to detect a correlated rectangle on a $64\times 64$ lattice. We consider four different sizes of the correlated rectangle: $5\times5$, $10\times10$, $20\times20$ and $40\times 40$. For each given size of the correlated region, we varied the correlation coefficient to capture the relationship between the power of our detection scheme and $\rho$. The results summarized in Figure \ref{DSCPP} are again based on $500$ runs for each combination of size and correlation coefficient of the correlated region. The observed effect of $A$ on its detectability is consistent with the results established in Theorem \ref{th:upper} and Theorem \ref{th:lower}: larger regions are easier to detect with the same correlation coefficient.

\begin{figure}[htbp]
\begin{center}
\includegraphics[width=\textwidth]{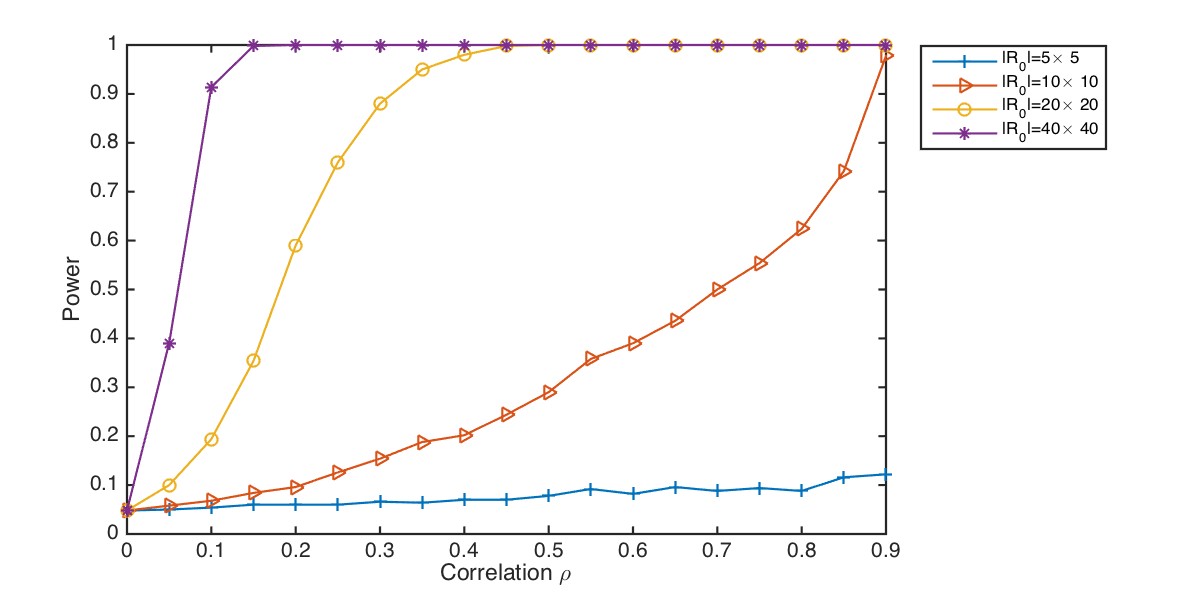}
\caption{Power plot for detecting a correlated rectangle of different sizes on a $64\times 64$ lattice.}
\label{DSCPP}
\end{center}
\end{figure}

Our final set of simulations is designed to assess the effect of $\II$. To this end we consider identifying a $10\times 10$ correlated rectangle on a squared lattice of size $32\times 32$, $64\times 64$, or $128\times 128$. As in the previous example, we repeat the experiment $500$ times for each combination of $\II$ and and a variety of values of $\rho$. The results are presented in Figure \ref{DICPP}. The observed effect of $|\II|$ is again consistent with our theoretical developments: as the size of lattice increases, detection becomes harder for a region of the same size and correlation.

\begin{figure}[htbp]
\begin{center}
\includegraphics[width=\textwidth]{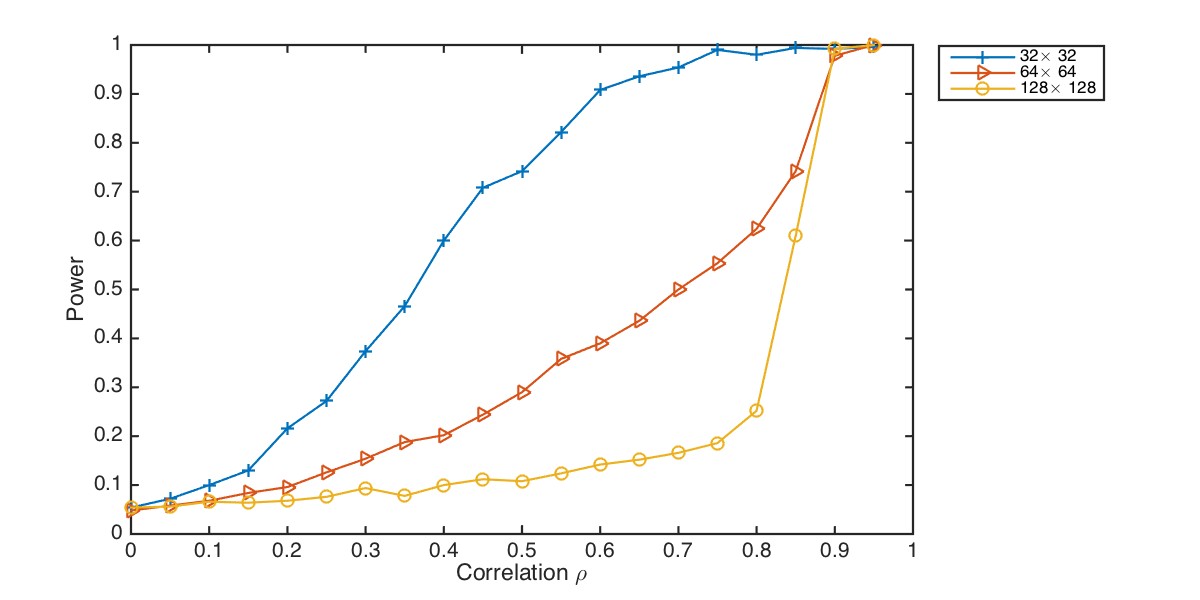}
\caption{Power plot for detecting a $10\times 10$ correlated rectangle on squared lattices of different sizes.}
\label{DICPP}
\end{center}
\end{figure}

\subsection{Real data example}

For illustration purposes, we now return to the image example we briefly mentioned in the introduction. This image originated from a concerted effort by several research groups to dissect the post-transcriptional process of human immunodeficiency virus type 1 (HIV-1) using imaging based approaches.

HIV uses the cell's mRNA nuclear export pathway to initiate the post-transcriptional stages of the viral life cycle. Nuclear export of a segment of the viral gRNA bearing the Rev Response Element (RRE) is essential to HIV gene regulation, viral replication and disease. It is well established that this process is regulated by a viral Rev trafficking protein that binds to the RRE and recruits the cellular CRM1 nuclear export receptor. In addition to nuclear export, Rev may also play a role in gRNA encapsidation and translation of gRNA. To gain insight into Rev's motility in the nucleus and cytoplasm to better understand Rev trafficking dynamics, as well as Rev's roles in viral gene expression and virus particle assembly live cell imaging was employed to monitor Rev and CRM1 behavior. It is expected that the REV/CRM1/RRE ribonuclear complex will have high colocalization in the cytoplasm during the viral life cycle. It has been shown that HIV-1 genomic RNAs (gRNAs) frequently exhibit ``burst" nuclear export kinetics. These events are characterized by striking en masse evacuations of gRNAs from the nucleus to flood the cytoplasm in conjunction with the onset of viral late gene expression. Burst nuclear export is regulated through interactions between the viral protein Rev, cellular nuclear export factor CRM1, and the gRNA's cis-acting RNA Rev response element (RRE). By monitoring mutant versions of the Rev protein unable to bind CRM1, export element deficient versions of the gRNA (RRE Minus), and lack of gRNA in the visual system, we can determine CRM1 trafficking behavior in the context of the virus. 

A specific data example is given in Figure \ref{fig:colocal} where dual-channel images of a wild type cell and a mutant cell are presented side-by-side. In each image, CRM1 is represented by green and the gRNA by red. While the ``burst'' nuclear export is visible for the wild type cell, it is less evident for the mutant cell. To further quantify such differences between the two cell types, we applied our method to this particular example following standard steps to preprocess the image: applying Otsu's method to each channel to remove background and then identifying spatial compartments where both channels are significantly expressed. On the post-processed images, we compute the test statistic $\tilde{T}^\ast$ and evaluate its corresponding p-value again by simulating the null distribution through 1000 Monte Carlo experiments. For the wild type cell, we obtained $\tilde{T}^\ast=5.65\times 10^3$ which is larger than any of the $1000$ values from the Monte Carlo simulations under the null hypothesis, suggesting a p-value$<0.1\%$, up to a Monte Carlo simulation error. On the other hand, the test statistic for the mutant cell is 8.98, which corresponds to a p-value of $0.846$.

\begin{figure}[ht!]
                \centering\includegraphics[width=8cm,height=5cm]{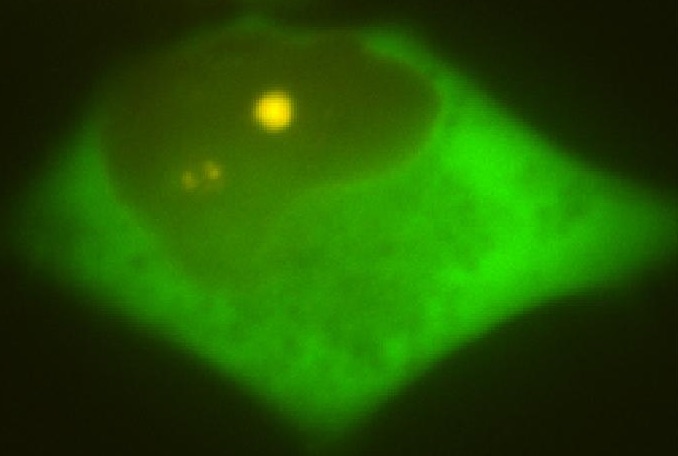}
                \centering\includegraphics[width=8cm,height=5cm]{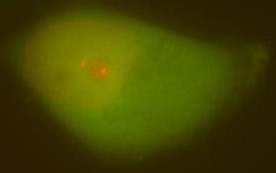}
                \caption{Colocalization between CRM1 and gRNA: comparison between the wild type (on the left) and mutant (on the right).}
                \label{fig:colocal}
\end{figure}

\section{Proofs of Main Results}

We now present the proofs to our main results, namely Theorems \ref{th:emp}, \ref{th:upper} and \ref{th:lower}. Proofs of Propositions \ref{pr:poly1} and \ref{pr:poly2}, as well as a number of auxiliary results, are relegated to the Appendix. To distinguish from the constants appeared in the previous sections, we shall use the capital letter $C$ to denote a generic positive constant that may take different values at each appearance.

\label{sec:proof}
\begin{proof} [Proof of Theorem \ref{th:emp}] We first prove the upper bound (\ref{eq:maxupper}) under conditions (\ref{eq:comp}) and (\ref{eq:entropy}). To this end, we shall establish a stronger result that there exists a constant $C>0$ such that for any $0<t<(\log n)^3$.
\begin{equation}
\label{eq:maxtail}
\PP\left\{\max_{R\in \Rcal(A)}L_R> 2\log n+C(\log \log n+t)\right\}\le \exp(-t).
\end{equation}
It is clear that (\ref{eq:maxupper}) follows immediately from (\ref{eq:maxtail}).

We now proceed to prove (\ref{eq:maxtail}). We shall consider the cases where $A\le (\log n)^5$ and $A\ge (\log n)^5$ separately. First consider the situation when $A\le (\log n)^5$. By Lemma \ref{lem:LR}, there exists a constant $C>0$ such that for any fixed $R\in \Rcal(A)$
$$
\PP\left\{L_R> x\right\}\le C\exp(-x/2).
$$
Applying union bound yields
\begin{eqnarray*}
\PP\left\{\max_{R\in \Rcal(A)}L_R> x\right\}&\le& C|\Rcal(A)|\exp(-x/2)\\
&\le& c_1CnA^{c_2}\exp(-x/2)\\
&\le& c_1Cn(\log n)^{5c_2}\exp(-x/2),
\end{eqnarray*}
where the second inequality follows from (\ref{eq:comp}). Equation (\ref{eq:maxtail}) then follows by taking
$$
x=2\log(c_1C)+2\log n+10c_2\log\log n+2t.
$$

The treatment for $A\ge (\log n)^5$ is more involved and we apply a chaining argument. Let $\Rcal_{\rm app}(A,e^{-s})$ be an $e^{-s}$ covering set of $\Rcal(A)$ so that
$$
|\Rcal_{\rm app}(A,e^{-s})|=N(A,e^{-s}).
$$
For any segment $R\in \Rcal(A)$, denote by
$$
\pi_s(R)=\argmin_{R'\in \Rcal_{\rm app}(A,e^{-s})}d(R,R').
$$
Of course, the minimizer on the right hand side may not be uniquely defined, in which case, we take $\pi_s(R)$ to be an arbitrarily chosen minimizer.

Write
$$
L_R=\sum_{s=s_\ast}^{s^\ast-1} \left(L_{\pi_{s+1}(R)}-L_{\pi_s(R)}\right)+(L_R-L_{\pi_{s^\ast}(R)})+L_{\pi_{s_\ast}(R)},
$$
where $s^\ast>s_\ast\ge \log \log(n/A)$ are to be specified later. It is clear that
\begin{equation}
\label{eq:chain}
\max_{R\in \Rcal(A)} L_R\le \sum_{s=s_\ast}^{s^\ast-1} \max_{R\in \Rcal(A)}\left|L_{\pi_{s+1}(R)}-L_{\pi_s(R)}\right|+\max_{R\in \Rcal(A)}\left|L_{R}-L_{\pi_{s^\ast}(R)}\right|+\max_{R\in \Rcal(A)}|L_{\pi_{s_\ast}(R)}|.
\end{equation}
We now bound the three terms on the right hand side of (\ref{eq:chain}) separately.

By definition,
$$
d(R, \pi_{s}(R))\le e^{-s},\qquad{\rm and}\qquad d(R, \pi_{s+1}(R))\le e^{-(s+1)}.
$$
Hence there exists a constant $C>0$ such that
$$
|\pi_{s}(R)\cap \pi_{s+1}(R)|\ge (1-C e^{-s})|R|,\qquad{\rm and}\qquad d(\pi_{s}(R),\pi_{s+1}(R))\le Ce^{-s}.
$$
Now by Lemma \ref{lem:dif}, for any fixed $R\in \Rcal(A)$,
$$
|L_{\pi_{s}(R)}-L_{\pi_{s+1}(R)}|\le C\left(e^{-s/2}x+|R|^{-1/2}x^{3/2}\right)
$$
with probability at least $1-Ce^{-x}$. An application of the union bound yields
\begin{eqnarray*}
&&\PP\left\{\max_{R\in \Rcal(A)}\left|L_{\pi_{s+1}(R)}-L_{\pi_s(R)}\right|> C\left(e^{-s/2}x+\sqrt{2}A^{-1/2}x^{3/2}\right)\right\}\\
&\le& CN(A,e^{-s})N(A,e^{-(s+1)})e^{-x}\\
&\le& C[N(A,e^{-(s+1)})]^2e^{-x}\\
&\le& c_4^2C\left({n\over A}\right)^2\left(\log{n\over A}\right)^{2c_5}e^{2c_6(s+1)}e^{-x},
\end{eqnarray*}
where the last inequality follows from (\ref{eq:entropy}). In particular, taking
$$x=t+2\log s+\log(c_4^2C)+2\log(n/A)+2c_5\log\log(n/A)+2c_6(s+1)$$
yields, with probability at least $1-s^{-2}e^{-t}$,
$$
\max_{R\in \Rcal(A)}\left|L_{\pi_{s+1}(R)}-L_{\pi_s(R)}\right|\le C\left((s+t+\log(n/A))e^{-s/2}+A^{-1/2}(s+t+\log(n/A))^{3/2}\right).
$$
Here we used the fact that $s\ge s_\ast\ge \log \log (n/A)$. Now applying the union bound over all $s_\ast\le s<s^\ast$, we get, with probability at least $1-s_\ast^{-1}e^{-t}\ge 1-e^{-t}$,
\begin{eqnarray*}
\sum_{s=s_\ast}^{s^\ast-1} \max_{R\in \Rcal(A)}\left|L_{\pi_{s+1}(R)}-L_{\pi_s(R)}\right|&\le& C\sum_{s=s_\ast}^{s^\ast-1}\left((s+t+\log(n/A))e^{-s/2}+A^{-1/2}(s+t+\log(n/A))^{3/2}\right)\\
&\le&C\left(s_\ast e^{-s_\ast/2}+A^{-1/2}(s^\ast)^{5/2}\right)\\
& &+C\left(e^{-s_\ast/2}(t+\log(n/A))+A^{-1/2}s^\ast (t+\log(n/A))^{3/2}\right).
\end{eqnarray*}

To bound the second term on the right hand side of (\ref{eq:chain}), we again apply Lemma \ref{lem:dif}. For any fixed $R\in \Rcal(A)$, we get
$$
\PP\left\{\left|L_R-L_{\pi_{s^\ast}(R)}\right|\ge C\left(e^{-s^\ast/2}x+\sqrt{2}A^{-1/2}x^{3/2}\right)\right\}\le Ce^{-x}.
$$
Another application of the union bound yields,
\begin{eqnarray*}
\max_{R\in \Rcal(A)}\left|L_R-L_{\pi_{s^\ast}(R)}\right|&\le& C\left(e^{-s^\ast/2}\log |\Rcal(A)|+A^{-1/2}(\log|\Rcal(A)|)^{3/2}+e^{-s^\ast/2}t+A^{-1/2}t^{3/2}\right)\\
&\le&c_2C\left(e^{-s^\ast/2}\log n+A^{-1/2}(\log n)^{3/2}+e^{-s^\ast/2}t+A^{-1/2}t^{3/2}\right),
\end{eqnarray*}
with probability at least $1-Ce^{-t}$, where we used (\ref{eq:comp}) in the last inequality.

Finally, for the third term on the right hand side of (\ref{eq:chain}), we have
$$
\PP\left\{\max_{R\in \Rcal_{\rm app}(A,e^{-s_\ast})} |L_R|\ge x\right\}\le CN(A,e^{-s_\ast})e^{-x/2}\le c_4C\left(n\over A\right)\left(\log{n\over A}\right)^{c_5}e^{c_6s_\ast}e^{-x/2}.
$$
Taking
$$
x=2\log(c_4C)+2\log{n\over A}+c_5\log\log{n\over A}+2c_6s_\ast+t
$$
yields, with probability at least $1-Ce^{-t}$,
$$
\max_{R\in \Rcal_{\rm app}(A,e^{-s_\ast})} |L_R|\le 2\log(c_4C)+2\log{n\over A}+c_5\log\log{n\over A}+2c_6s_\ast+t.
$$

In summary, we get, with probability at least $1-Ce^{-t}$,
\begin{eqnarray*}
\max_{R\in \Rcal(A)} L_R&\le& C\biggl(s_\ast e^{-s_\ast/2}+A^{-1/2}(s^\ast)^{5/2}+e^{-s_\ast/2}t+A^{-1/2}s^\ast t^{3/2}+e^{-s^\ast/2}\log n\\
&&+A^{-1/2}(\log n)^{3/2}+e^{-s^\ast/2}t+A^{-1/2}t^{3/2}+e^{-s_\ast/2}\log{n\over A}+A^{-1/2}s^\ast (\log({n/A}))^{3/2}\biggr)\\
&&+2\log(c_4C)+2\log{n\over A}+c_5\log\log{n\over A}+2c_6s_\ast+t.
\end{eqnarray*}
Recall that $A\ge (\log n)^5$. If we take $s^\ast=2\log n$ and $s_\ast=2\log\log (n/A)$, then for any $t\le (\log n)^3$, we can deduce from the above inequality that
\begin{equation}
\label{eq:longup}
\max_{R\in \Rcal(A)} L_R\le 2\log{n\over A}+C\left(\log\log{n\over A}+t\right),
\end{equation}
which implies (\ref{eq:maxtail}).

We now prove (\ref{eq:maxeq}) if in addition, (\ref{eq:entropy0}) holds. In the light of (\ref{eq:entropy}), we can find a subset $\tilde{\Rcal}(A)$ of $\Rcal(A)$ such that for any $R_1,R_2\in \tilde{\Rcal}(A)$, $R_1\cap R_2=\emptyset$ and
$$
|\tilde{\Rcal}(A)|\ge c_3{n\over A}.
$$
Obviously,
$$
\max_{R\in \Rcal(A)}L_R\ge \max_{R\in \tilde{\Rcal}(A)}L_R.
$$
If $A\le (\log n)^5$, then
\begin{eqnarray*}
\PP\left\{\max_{R\in \tilde{\Rcal}(A)}L_R\le x\right\}&=&\prod_{R\in \tilde{\Rcal}(A)}\PP\{L_R\le x\}\\
&\le& \prod_{R\in \tilde{\Rcal}(A)}\left(1-C|R|^{-1/2}e^{-x/2}\right)\\
&\le& \left[1-CA^{-1/2}e^{-x/2}\right]^{c_3n/A}\\
&\le& \left[1-C(\log n)^{-5/2}e^{-x/2}\right]^{c_3n/A},
\end{eqnarray*}
where the first inequality follows from the lower bound given by Lemma \ref{lem:LR}. It can then be derived that
\begin{equation}
\label{eq:shortlow}
\max_{R\in \tilde{\Rcal}(A)}L_R\ge 2\log n+O_p(\log \log n).
\end{equation}
Together with (\ref{eq:maxupper}), (\ref{eq:shortlow}) implies the desired claim when $A\le (\log n)^5$.

Next we consider the case when $A\ge (\log n)^5$. We proceed in a similar fashion as before but rely on the following tail bound of $L_R$: if $A\ge 24$, then there exists a constant $C>0$ such that for any $R\in \Rcal(A)$ and $0<x<\sqrt{A}$,
\begin{equation}
\label{eq:tailev}
\PP\{L_R\ge x\}\le Cx^{-1/2}\exp(-x/2).
\end{equation}
If (\ref{eq:tailev}) holds, then
$$
\PP\left\{\max_{R\in \tilde{\Rcal}(A)}L_R\le x\right\}\ge \left(1-Cx^{-1/2}e^{-x/2}\right)^{c_3n/A},
$$
which yields
$$
\max_{R\in {\Rcal}(A)}L_R\ge \max_{R\in \tilde{\Rcal}(A)}L_R\ge 2\log(n/A)+O_p(\log\log(n/A)).
$$
Together with (\ref{eq:maxupper}), this concludes the proof.

It now remains to prove (\ref{eq:tailev}). Write
$$T_R=(|R|-2)r_R^2.$$
Note that $\log(1+x)>x-x^2/2$ for any $x>0$. We get
$$
L_R\ge(|R|-2)\log(1+r_R^2)\ge T_R^2-{T_R^4\over 2(|R|-2)}\ge T_R^2-{T_R^4\over A-4},
$$
for any $A\ge 5$, where in the last inequality we used the fact that $|R|>A/2$ for any $R\in \Rcal(A)$. This can be further lower bounded by $T_R^2-{3T_R^4/A}$ for any $A\ge 6$. Thus, for any $0<x<A/24$,
\begin{eqnarray*}
\PP\{L_R\ge x\}&\ge& \PP\left\{T_R^2-{3T_R^4\over A}\ge x\right\}\\
&\ge& \PP\left\{T_R^2-{3T_R^4\over A}\in [x,2x)\right\}\\
&\ge& \PP\left\{T_R^2\in [x+12x^2/A,2x+3x^2/A)\right\}\\
&\ge& \PP\left\{T_R\in [\sqrt{x+12x^2/A},\sqrt{2x+3x^2/A})\right\}.
\end{eqnarray*}
Because $T_R\sim t_{|R|-2}$, we have
\begin{eqnarray*}
&&\PP\left\{T_R\in [\sqrt{x+12x^2/A},\sqrt{2x+3x^2/A})\right\}\\
&\ge& C\int_{\sqrt{x+12x^2/A}}^{\sqrt{2x+3x^2/A}}\left(1+\frac{u^{2}}{|R|-2}\right)^{-\frac{|R|-1}{2}}du\\
&\ge& C\int_{\sqrt{x+12x^2/A}}^{\sqrt{2x+3x^2/A}}\exp\left[{-\frac{|R|-1}{2}}\log\left(1+\frac{u^{2}}{|R|-2}\right)\right]du\\
&\ge&C\int_{\sqrt{x+12x^2/A}}^{\sqrt{2x+3x^2/A}}\exp\left({-\frac{|R|-1}{2(|R|-2)}}u^{2}\right)du,
\end{eqnarray*}
for some constant $C>0$, where in the last inequality we used the fact that $\log(1+x)\le x$ for all $x\ge 0$. Thus,
\begin{eqnarray*}
&&\PP\left\{T_R\in [\sqrt{x+12x^2/A},\sqrt{2x+3x^2/A})\right\}\\
&\ge& C(2x+3x^2/A)^{-1/2}\int_{\sqrt{x+12x^2/A}}^{\sqrt{2x+3x^2/A}}u\exp\left({-\frac{|R|-1}{2(|R|-2)}}u^{2}\right)du\\
&=&C(2x+3x^2/A)^{-1/2}(1-(|R|-1)^{-1})\biggl[\exp\left({-\frac{|R|-1}{2(|R|-2)}}(x+12x^2/A)\right)\\
&&\hskip 100pt -\exp\left({-\frac{|R|-1}{2(|R|-2)}}(2x+3x^2/A)\right)\biggr].
\end{eqnarray*}
Recall that $0<x<A/24$. We get
\begin{eqnarray*}
&&\PP\left\{T_R\in [\sqrt{x+12x^2/A},\sqrt{2x+3x^2/A})\right\}\\
&\ge& C x^{-1/2}\exp\left({-\frac{|R|-1}{2(|R|-2)}}(x+12x^2/A)\right)\\
&\ge& C x^{-1/2}\exp\left({-\frac{A-1}{2(A-2)}}(x+12x^2/A)\right)\\
&\ge& C x^{-1/2}\exp(-x/2),
\end{eqnarray*}
where in the last inequality, we used the fact that $x\le \sqrt{A}$. The proof is then completed.
\end{proof}
\vskip 25pt

\begin{proof}[Proof of Theorem \ref{th:upper} (Consistency of $T^\ast$)]
We first show that the claim is true for $T^\ast$. To this end, we begin by arguing that $q_\alpha=O(1)$, and then show that under $H_1$, ${T}^\ast\to \infty$. Note that
\begin{eqnarray*}
T^\ast&=&\max_{R\in \Rcal} \left\{{1\over \log\log(n/|R|)}\left[L_R-2\log\left({n\over |R|}\right)\right]\right\}\\
&=&\max_{1\le k\le \log n}\max_{R\in \Rcal(e^{-k+1}n)} \left\{{1\over \log\log(n/|R|)}\left[L_R-2\log\left({n\over |R|}\right)\right]\right\}.
\end{eqnarray*}
As shown in the proof of Theorem \ref{th:emp}, there exists a constant $C>0$ such that for any $0<t<(\log n)^3$,
$$
\PP\left\{\max_{R\in \Rcal(e^{-k+1}n)} L_R\ge 2k+C\left(\log k+t\right)\right\}\le \exp(-t).
$$
Taking $t=x+\log (2k^2)$ yields
$$
\PP\left\{\max_{R\in \Rcal(e^{-k+1}n)} \left\{{1\over \log\log(n/|R|)}\left[L_R-2\log\left({n\over |R|}\right)\right]\right\}\ge C(x+1)\right\}\le {1\over 2k^2}\exp(-x).
$$
Applying union bound over all $k$, we get
$$
\PP\left\{T^\ast\ge C(x+1)\right\}\le \sum_{1\le k\le \log n} {1\over 2k^2}\exp(-x)\le \exp(-x),
$$
which implies that $q_\alpha\le C(1-\log(1-\alpha))$.

It now suffices to show that if (\ref{eq:upper}) holds for some $R\in \Rcal$, then $T^\ast\to \infty$. To this end, note that
$$
T^\ast\ge {1\over \log\log(n/|R|)}\left[L_R-2\log\left({n\over |R|}\right)\right].
$$
We treat the case $|R|\ge \log n$ and $|R|\le \log n$ separately.

Consider first the situation when $|R|\le \log n$. By Lemma \ref{lem:LRH1},
$$
\left({1-r_R^2\over 1-\rho^2}\right)\sum_{i\in R}(Y_i-\bar{Y}_R )^2\sim \chi^2_{|R|-2}.
$$
Applying the $\chi^2$ tail bounds of \cite{ChiSqu}, we get, with probability at least $1-2e^{-x}$,
$$
\left({1-r_R^2\over 1-\rho^2}\right)\sum_{i\in R}(Y_i-\bar{Y}_R )^2\le (|R|-2)+2\sqrt{x(|R|-2)}+2x
$$
and
$$
\sum_{i\in R}(Y_i-\bar{Y}_R )^2\ge (|R|-1)-2\sqrt{x(|R|-1)}.
$$
Under this event,
$$
{1-r_R^2\over 1-\rho^2}\le {(|R|-2)+2\sqrt{x(|R|-2)}+2x\over (|R|-1)-2\sqrt{x(|R|-1)}}.
$$
Assuming that $x=o(|R|)$, this can be further simplified as
$$
{1-r_R^2\over 1-\rho^2}\le 1+o\left(\sqrt{x\over |R|}\right).
$$
If in addition, $x\to\infty$, then
\begin{eqnarray*}
-(|R|-2)\log{(1-r_{R}^{2})}&\ge& -|R|\log{(1-\rho^{2})}+o\left(\sqrt{x|R|}\right)\\
&\ge& 2\log(n/|R|)+\delta_n\log(n/|R|)+o\left(\sqrt{x|R|}\right),
\end{eqnarray*}
which diverges with $n$ because
$$
\delta_n\log(n/|R|)\gg \sqrt{\log(n/|R|)}\gg |R|\gg \sqrt{x|R|}.
$$
Since
$$
\delta_n\log(n/|R|)\gg \sqrt{\log(n/|R|)}\gg \log\log (n/|R|),
$$
this immediately suggests that
$$
T^\ast\ge {1\over \log\log(n/|R|)}\left[L_R-2\log\left({n\over |R|}\right)\right]\to_p\infty.
$$

Next consider the case when $|R|\ge \log n$. Assume without loss of generality that $\rho>0$. The treatment for $\rho<0$ is identical. Following an argument similar to that for Lemma \ref{lem:r}, we get
$$
\sum_{i\in R}(X_{i}-\bar{X}_R )^2, \sum_{i\in R}(Y_{i}-\bar{Y}_R )^2\le (|R|-1)+2\sqrt{x(|R|-1)}+2x
$$
and
$$
\sum_{i\in R}(X_i-\bar{X}_R )(Y_i-\bar{Y}_R )\ge (|R|-1)\rho-2\sqrt{x(|R|-1)}-2x.
$$
with probability at least $1-6e^{-x}$. Denote this event by $\Ecal(x)$. We shall now proceed under $\Ecal(x)$ with an appropriately chosen $x\to\infty$. 
\begin{equation}
\label{eq:bdrrbyrho}
r_R\ge {\rho-2\sqrt{x/(|R|-1)}-2[x/(|R|-1)]\over 1+2\sqrt{x/(|R|-1)}+2[x/(|R|-1)]}.
\end{equation}
It is not hard to see that under the condition (\ref{eq:upper}), $|R|\rho^2\to \infty$. Assuming that $x\to\infty$ such that $x=o(|R|\rho^2)$, we get
$$
r_R\ge \rho+o\left(\sqrt{x\over |R|}\right)
$$
Then,
\begin{equation}
\label{eq:rhobd}
L_R\ge -(|R|-2)\log\left[1-\left(\rho+o\left(\sqrt{x\over |R|}\right)\right)^2\right]
\end{equation}
Recall that
$$
-|R|\log(1-\rho^2)\ge (2+\delta_n)\log\left({n\over |R|}\right).
$$
Denote by $\rho_\ast>0$ the solution to
$$
-|R|\log(1-y^2)=(2+\delta_n)\log\left({n\over |R|}\right).
$$
It is clear that $\rho\ge \rho_\ast$. Together with the fact that the right hand side of (\ref{eq:rhobd}) is an increasing function of $\rho$, we get
\begin{eqnarray*}
L_R&\ge& -(|R|-2)\log\left[1-\left(\rho_\ast+o\left(\sqrt{x\over |R|}\right)\right)^2\right]\\
&=&-(|R|-2)\log(1-\rho_\ast^2)+o(\sqrt{x|R|}\rho_\ast)\\
&=&(2+\delta_n)\log\left({n\over |R|}\right)+o(\sqrt{x|R|\rho_\ast^2}).
\end{eqnarray*}
Note that
$$
|R|\rho_\ast^2\le 2|R|\log(1+\rho^2_\ast)\le -2|R|\log(1-\rho^2_\ast)= 2(2+\delta_n)\log\left({n\over |R|}\right).
$$
It is not hard to see that
$$
{\delta_n^2\over 2+\delta_n}\log\left({n\over |R|}\right)\to \infty
$$
if (\ref{eq:upper}) holds. Assuming that
$$
x=o\left({\delta_n^2\over 2+\delta_n}\log\left({n\over |R|}\right)\right),
$$
we get
$$
T^\ast\ge (\log\log(n/|R|))^{-1}(L_R-2\log(n/|R|))\to \infty.
$$
This concludes the proof of consistency of $T^\ast$ under (\ref{eq:upper}).
\end{proof}
\vskip 25pt

\begin{proof}[Proof of Theorem \ref{th:upper} (Consistency of $\tilde{T}^\ast$)] We now consider the computationally efficient test based on $\tilde{T}^\ast$ is also consistent. As before, we begin by arguing that $\tilde{q}_\alpha=O(1)$, and then show that under $H_1$, $\tilde{T}^\ast\to \infty$. To show that $\tilde{q}_\alpha=O(1)$, it suffices to note that
\begin{eqnarray*}
T^\ast&=&\max_{R\in \Rcal} \left\{(\log\log(n/|R|))^{-1}\left(L_R-2\log(n/|R|)\right)\right\}\\
&\ge&\max_{R\in \cup_k\tilde{\Rcal}_k} \left\{(\log\log(n/|R|))^{-1}\left(L_R-2\log(n/|R|)\right)\right\}\\
&=&\tilde{T}^\ast.
\end{eqnarray*}
Therefore, $\tilde{q}_\alpha\le q_\alpha=O(1)$ following the argument before.

Next we show that under the alternative hypothesis where $X_i$ and $Y_i$ are correlated on a set $R\in \Rcal_k$ for some $k$, $\tilde{T}^\ast\to \infty$. By definition, there exists a $\tilde{R}\in \tilde{\Rcal}_k$ such that
\begin{equation}
\label{eq:RRtilde}
d(R,\tilde{R})\le {1\over 4k^2}.
\end{equation}
Observe that
$$
\tilde{T}^\ast\ge \tilde{T}^\ast_k\ge (\log\log(n/|\tilde{R}|))^{-1}\left(L_{\tilde{R}}-2\log(n/|\tilde{R}|)\right).
$$
It now suffices to show that the rightmost hand side is unbounded with probability approaching to 1 when $k\le k_\ast$. To this end, we first consider the case when $\tilde{R}\subseteq R$.

Note that if $\tilde{R}\subseteq R$, then (\ref{eq:model}) holds for any $i\in \tilde{R}$. Following an identical argument for consistency of $T^\ast$, it suffices to show that there exists a $\tilde{\delta}_n>0$ such that
\begin{equation}
\label{eq:condtild1}
\tilde{\delta}_n\sqrt{\log(n/|\tilde{R}|)}\to \infty
\end{equation}
and
\begin{equation}
\label{eq:condtild2}
-|\tilde{R}|\log(1-\rho^2)\ge (2+\tilde{\delta}_n)\log\left({n\over |\tilde{R}|}\right).
\end{equation}
Observe that (\ref{eq:RRtilde}) implies that
$$
|\tilde{R}|\ge \left(1-{1\over 4k^2}\right)|R|.
$$
Thus
\begin{eqnarray*}
|\tilde{R}|\log{1\over 1-\rho^2}&\ge& \left(1-{1\over 4k^2}\right)|R|\log{1\over 1-\rho^2}\\
&\ge& \left(1-{1\over 4k^2}\right)(2+\delta_n)\log\left({n\over |R|}\right).
\end{eqnarray*}
Because
$$
\log\left({n\over |R|}\right)=\log\left({n\over |\tilde{R}|}\right)+\log\left({|\tilde{R}|\over |R|}\right)\ge \log\left({n\over |\tilde{R}|}\right)+\log\left(1-{1\over 4k^2}\right)\ge \log\left({n\over |\tilde{R}|}\right)-{1\over 4k^2},
$$
we get
\begin{eqnarray*}
|\tilde{R}|\log{1\over 1-\rho^2}&\ge& \left(1-{1\over 4k^2}\right)(2+\delta_n)\left[\log\left({n\over |\tilde{R}|}\right)-{1\over 4k^2}\right]\\
&\ge& \left(1-{1\over 4k^2}\right)^2(2+\delta_n)\log\left({n\over |\tilde{R}|}\right)\\
&\ge&\left(1-{1\over 2k^2}\right)(2+\delta_n)\log\left({n\over |\tilde{R}|}\right).
\end{eqnarray*}
Let
$$
\tilde{\delta}_n=\left(1-{1\over 2k^2}\right)\delta_n-{1\over k^2}.
$$
Then (\ref{eq:condtild2}) holds. We now verify (\ref{eq:condtild1}). Recall that
$$
\delta_n^2(k-1)\log 2\le \delta_n^2\log\left({n\over |R|}\right)\to \infty,
$$
we get, for sufficiently large $n$,
$$
\tilde{\delta}_n\ge {1\over 4}\delta_n.
$$
This implies that
$$
\tilde{\delta}_n^2\log\left({n\over |\tilde{R}|}\right)\ge \tilde{\delta}_n^2\log\left({n\over |R|}\right)\ge {1\over 16}\delta_n^2\log\left({n\over |R|}\right)\to \infty,
$$
which completes the proof for the case $\tilde{R}\subseteq R$.

Now consider the case when $\tilde{R}\not\subseteq R$. By definition,
$$
{|\tilde{R}\cap R|\over \sqrt{|R||\tilde{R}|}}\ge 1-{1\over 4k^2}.
$$
Because $\tilde{R}\cap R\subseteq \tilde{R}$, we get
\begin{equation}
\label{eq:capRsize1}
{|\tilde{R}|\over |R|}\ge \left(1-{1\over 4k^2}\right)^2.
\end{equation}
Thus,
$$
|\tilde{R}\cap R|\ge \left(1-{1\over 4k^2}\right)^{3/2}|R|\ge \left(1-{1\over 3k^2}\right)|R|.
$$
Similarly, we can derive that
\begin{equation}
\label{eq:capRsize}
|\tilde{R}\cap R|\ge \left(1-{1\over 3k^2}\right)|\tilde{R}|.
\end{equation}
Following the same treatment as for the previous case, we can derive that
$$
{1\over \log\log (n/ |\tilde{R}\cap R|)}\left[L_{\tilde{R}\cap R}-2\log\left({n\over |\tilde{R}\cap R|}\right)\right]\to_p \infty.
$$
Since $|\tilde{R}\cap R|\le |\tilde{R}|$,
$$
{1\over \log\log {n\over |\tilde{R}|}}\left[L_{\tilde{R}\cap R}-2\log\left({n\over |\tilde{R}|}\right)\right]\ge {1\over \log\log {n\over |\tilde{R}\cap R|}}\left[L_{\tilde{R}\cap R}-2\log\left({n\over |\tilde{R}\cap R|}\right)\right]\to\infty.
$$
It now suffices to show that
\begin{equation}
\label{eq:finalTtilde}
|L_{\tilde{R}\cap R}-L_{\tilde{R}}|=O_p\left(\log\log \left(n\over |\tilde{R}|\right)\right)
\end{equation}
In the light of (\ref{eq:capRsize1}),
\begin{eqnarray*}
\log\log \left(n\over |\tilde{R}|\right)&\ge& \log\left[\log \left(n\over |R|\right) +2\log\left(1-{1\over 4k^2}\right)\right]\\
&\ge& \log\left[(k-1)\log 2-{1\over 2k^2}\right]=O(\log k).
\end{eqnarray*}
On the other hand, by Lemma \ref{lem:dif1},
$$
|L_{\tilde{R}\cap R}-L_{\tilde{R}}|\le C\left[{1\over 3k^2}x+|\tilde{R}|^{-1/2}x^{3/2}\right],
$$
with probability at least $1-e^{-x}$. Observe that
$$
|\tilde{R}|\ge \left(1-{1\over 4k^2}\right)^2|R|\ge \left(1-{1\over 2k^2}\right)|R|\ge n2^{-(k+1)}.
$$
Equation (\ref{eq:finalTtilde}) then follows by taking
$$
x=\min\left\{k^2\log k,2^{-k/3}n^{1/3}(\log k)^{2/3}\right\}.
$$
The proof is now completed.
\end{proof}
\vskip 25pt

\begin{proof}[Proof of Theorem \ref{th:lower}]
Our argument is similar to those used earlier by \cite{lepski2000asymptotically} and \cite{walther2010optimal}. We shall outline only the main steps for brevity. Note first that a lower bound for a special case necessarily yields a lower bound for the general case. Thus it suffices to consider the case when $\mu_{1}=\mu_{2}=0$ and $\sigma_{1}=\sigma_{2}=1$. In the light of (\ref{eq:entropy0}), for any $A$, we can find $\tilde{\Rcal}(A)\subset \Rcal(A)$ such that $|\tilde{\Rcal}(A)|=c_3(n/A)$, and for any $R_1, R_2\in \tilde{\Rcal}(A)$, $R_1\cap R_2=\emptyset$. For brevity, we shall assume that $c_3=1$ and for any $R\in \tilde{\Rcal}(A)$, $|R|=A$. More general case can be treated in the same fashion albeit the argument becomes considerably more cumbersome.

Denote by $\PP_{0}$ the joint distribution of $\{(X_i,Y_i): i\in \II\}$ under null hypothesis, and by $\PP_{R}$ the joint distribution under alternative hypothesis where $X_i$ and $Y_i$ are correlated on $R\in \tilde{\Rcal}(A)$ so that (\ref{eq:model}) holds for $i\in R$ and (\ref{eq:modelh0}) holds for $i\notin R$. The likelihood ratio between $\PP_0$ and $\PP_R$ can be computed:
$$
W_R=\frac{d\PP_R}{d\PP_0}=\frac{1}{(1-\rho^{2})^{A/2}}\exp\left\{-\frac{\sum_{i\in R}\left(\rho^{2}X_i^{2}-2\rho X_{i}Y_{i}+\rho^{2}Y_{i}^{2}\right)}{2(1-\rho^{2})}\right\}
$$
To prove the first statement, we first show
$$\EE_{{0}}(W_R^{1+\delta_n/4})/(\eta |\tilde{\Rcal}(A)|)^{\delta_n/4}\rightarrow 0\qquad\mbox{ for any } 0<\eta<1,$$
where $\EE_0$ stands for expectation taken with respect to $\PP_0$.

It can be computed that
$$\EE_{{0}}(W_R^{1+\delta_n/4})=\frac{1}{(1-\rho^{2}\delta_n^{2}/16)^{A/2}(1-\rho^{2})^{A\delta_n/8}}.$$
Recall that
$$A\log{1\over 1-\rho^{2}}\le (2-\delta_n)\log{n\over A}.$$
We get
\begin{align*}
&-\log\left[\EE_{0}(W_R^{1+\delta_n/4})/(\eta |\tilde{\Rcal}(A)|)^{\delta_n/4}\right]\\
\ge & \frac{A\delta_n}{8}\log(1-\rho^{2})+\frac{A}{2}\log(1-\rho^{2}\delta_n^{2}/16)+{\delta_n\over 4}\log{n\over A}+{\delta_n\over 4}\log{\eta}\\
\ge & \frac{\delta_n^2}{8}\log{\frac{n}{A}}-(1-\delta_n/2)\left(\log{\frac{n}{A}}\right)\frac{\log(1-\rho^{2}\delta_n^{2}/16)}{\log(1-\rho^{2})}+{\delta_n\over 4}\log{\eta}\\
\ge &\frac{\delta_n^2}{8}\log{\frac{n}{A}}-{\delta_n^{2}\over 16}(1-\delta_n/2)\left(\log{\frac{n}{A}}\right)+{\delta_n\over 4}\log{\eta}\\
= &\frac{\delta_n^2}{16}(1+\delta_n/2)\log{\frac{n}{A}}+{\delta_n\over 4}\log{\eta}\\
\ge & \frac{\delta_n^2}{16}\log{\frac{n}{A}}\to \infty.
\end{align*}
Thus,
$$
\EE_{0}(W_R^{1+\delta_n/4})/(\eta |\tilde{\Rcal}(A)|)^{\delta_n/4}\rightarrow 0.
$$

Next, we argue that
$$\EE_{0}\left||\tilde{\Rcal}(A)|^{-1}\sum_{R\in \tilde{\Rcal}(A)}W_R-1\right|\rightarrow 0.$$
To this end, write
$$
\bar{W}=|\tilde{\Rcal}(A)|^{-1}\sum_{R\in \tilde{\Rcal}(A)}(W_R-1),
$$
$$
\bar{W}_1=|\tilde{\Rcal}(A)|^{-1}\sum_{R\in \tilde{\Rcal}(A)}(W_R-1){\bf 1}_{(|W_R-1|>\eta |\tilde{\Rcal}(A)|)},
$$
and
$$
\bar{W}_2=|\tilde{\Rcal}(A)|^{-1}\sum_{R\in \tilde{\Rcal}(A)}(W_R-1){\bf 1}_{|W_R-1|\le \eta |\tilde{\Rcal}(A)|}.
$$
Observe that
$$
\EE_{0}\left|\bar{W}\right|\le \EE_{0}\left|\bar{W}_1\right|+\EE_{0}\left|\bar{W}_2\right|\le \EE_{0}\left|\bar{W}_1\right|+\eta.
$$
On the other hand,
$$
\EE_{0}\left|\bar{W}_1\right|\le  \EE_{0}(W_R{\bf 1}_{(W_R>\eta |\tilde{\Rcal}(A)|)})\le \EE_{0}(W_R^{1+\delta_n/4})/(\eta |\tilde{\Rcal}(A)|)^{\delta_n/4}\to 0.
$$
We can take $\eta\downarrow 0$ to get
$$\EE_{0}\left||\tilde{\Rcal}(A)|^{-1}\sum_{R\in \tilde{\Rcal}(A)}W_R-1\right|\rightarrow 0.$$

Finally, let $\PP_1$ be the uniform mixture of $\PP_R$ for $R\in \tilde{\Rcal}(A)$, that is,
$$\PP_1=|\tilde{\Rcal}(A)|^{-1}\sum_{R\in \tilde{\Rcal}(A)} \PP_{R}.$$
Then for any test $\Delta$,
\begin{align*}
\PP_{0}(\Delta=1)+\PP_{1}(\Delta=0)=&\EE_{0}(\Delta)+1-\min_{R\in \tilde{\Rcal}(A)}\EE_{R}(\Delta)\\
\ge&\EE_{0}(\Delta)+1-|\tilde{\Rcal}(A)|\sum_{R\in \tilde{\Rcal}(A)}\EE_{R}(\Delta)\\
\ge &1-\EE_{0}(\Delta(1-|\tilde{\Rcal}(A)|^{-1}\sum_{R\in \tilde{\Rcal}(A)}W_{R}))\\
\ge &1-\EE_{0}\left|1-|\tilde{\Rcal}(A)|^{-1}\sum_{R\in \tilde{\Rcal}(A)}W_{R}\right|\rightarrow 1,
\end{align*}
which completes the proof of the first statement.

To show the second statement, we assume the contrary that $c_{n}$ is bounded from above. Then $\{c_n\}$ must have a convergent subsequence. Without loss of generality, assume $c_{n}$ itself converges to some $b\in [0,\infty)$. Then
$$
\log{W_R}\rightarrow_{d} N\left(-\frac{b}{2},b\right),
$$
which implies that
$$
\lim\sup\PP_{R}(\Delta=1)< 1.
$$
This contradicts with the fact that that the type II error of $\Delta$ goes to 0 as $n\to \infty$. The second statement is therefore proven.
\end{proof}


\section*{Appendix -- Auxiliary Results and Proofs}
We first state tail bounds for $t$ and $F$ distributions necessary for our derivations.

\begin{lemma}
\label{lem:tconcen}
Let $X$ be a random variable following a $t$ distribution with degree of freedom $n>1$. There exists a numerical constant $0<c_1<c_2$ such that
\begin{equation}
\label{eq:ttail}
c_1n^{-1/2}\left(1+\frac{x^{2}}{n}\right)^{-\frac{n}{2}}\le \PP(|X|>x)\le c_2\left(1+\frac{x^{2}}{n}\right)^{-\frac{n}{2}}
\end{equation}
for any $x\ge 1$. In particular,
\begin{equation}
\label{eq:lrtail}
c_1n^{-1/2}e^{-u/2}\le \PP\left\{n\log\left(1+{X^2\over n}\right)\ge u\right\}\le c_2e^{-u/2},
\end{equation}
for any $u\ge 1$.
\end{lemma}

\begin{proof}[Proof of Lemma \ref{lem:tconcen}]
Recall that the density of a $t$ distribution with degree of freedom $n$ is
$$
f(x)=\frac{\Gamma(\frac{n+1}{2})}{\sqrt{n\pi}\Gamma(\frac{n}{2})}\left(1+\frac{x^{2}}{n}\right)^{-\frac{n+1}{2}}\le C\left(1+\frac{x^{2}}{n}\right)^{-\frac{n+1}{2}},
$$
for an absolute constant $C>0$. Then, for any $u>0$,
\begin{align*}
\PP(X>u)&\le C\int_{u}^{\infty}\left(1+\frac{x^{2}}{n}\right)^{-\frac{n+1}{2}}dx\\
&\le C\int_{u}^{\infty}{x\over u}\left(1+\frac{x^{2}}{n}\right)^{-\frac{n+1}{2}}dx\\
&= {nC\over 2u}\int_{u}^{\infty}\left(1+\frac{x^{2}}{n}\right)^{-\frac{n+1}{2}}d\left(1+\frac{x^{2}}{n}\right)\\
&= {nC\over (n-1)u}\left(1+\frac{x^{2}}{n}\right)^{-\frac{n-1}{2}}\biggl|_u^\infty\\
&\le 2C {1\over u}\left(1+\frac{u^{2}}{n}\right)^{-\frac{n-1}{2}}.
\end{align*}
The upper bound in (\ref{eq:ttail}) follows immediately by taking $c=4\sqrt{2}C$, by symmetry of $t$ distribution. On the other hand, observe that
$$
f(x)\ge C\left(1+\frac{x^{2}}{n}\right)^{-\frac{n+1}{2}},
$$
for some constant $C>0$. Thus,
\begin{align*}
\PP(X>u)&\ge C\int_{u}^{\infty}\left(1+\frac{x^{2}}{n}\right)^{-\frac{n+1}{2}}dx\\
&\ge C\int_{u}^{\infty}{x\over \sqrt{n}}\left(1+\frac{x^{2}}{n}\right)^{-\frac{n}{2}-1}dx\\
&= {\sqrt{n}C\over 2}\int_{u}^{\infty}\left(1+\frac{x^{2}}{n}\right)^{-\frac{n}{2}-1}d\left(1+\frac{x^{2}}{n}\right)\\
&= {\sqrt{n}C\over (n-1)}\left(1+\frac{x^{2}}{n}\right)^{-\frac{n}{2}}\biggl|_u^\infty\\
&\ge {\sqrt{n}C\over (n-1)}\left(1+\frac{u^{2}}{n}\right)^{-\frac{n}{2}}.
\end{align*}
The lower bound in (\ref{eq:ttail}) then follows immediately.

Now, taking
$$
x=\sqrt{n(e^{u/n}-1)}
$$
in (\ref{eq:ttail}) yields (\ref{eq:lrtail}).
\end{proof}
\vskip 25pt

\begin{lemma}
\label{lem:ftail}
Let $U_1\sim \chi^2_{n_1}$ and $U_2\sim \chi^2_{n_2}$ be two independent random variables. Then for any $-1<x<1$,
$$
\PP\left\{\left|{n_1+n_2\over n_1}{U_1\over U_1+U_2}-1\right|\ge x\right\}\le 2\exp\left(-{n_1x^2\over 12}\right).
$$
\end{lemma}

\begin{proof}[Proof of Lemma \ref{lem:ftail}] As shown by \cite{dasgupta2003elementary}, for any $x>0$,
$$
\PP\left\{{n_1+n_2\over n_1}{U_1\over U_1+U_2}\le 1-x\right\}\le \exp\left({n_1\over 2}(x+\log (1-x))\right),
$$
and
$$
\PP\left\{{n_1+n_2\over n_1}{U_1\over U_1+U_2}\ge 1+x\right\}\le \exp\left({n_1\over 2}(-x+\log (1+x))\right).
$$
The claim then follows from the fact that
$$
\log(1+x)\le x-{x^2\over 6}
$$
for all $x$ such that $|x|<1$.
\end{proof}
\vskip 25pt

The following observation on the sample correlation coefficient is useful:

\begin{lemma}
\label{lem:LRH1}
Assume that $\{(X_i,Y_i): i\in R\}$ are iid copies of $(X,Y)\sim N((\mu_1,\mu_2)^\top,\Sigma)$ where
$$\Sigma=\begin{bmatrix}1 & \rho \\ \rho  & 1\end{bmatrix}.$$
Then
$$\sum_{i\in R}(Y_{i}-\bar{Y}_R )^{2}(1-r_{R}^{2})\sim (1-\rho^{2})\chi^{2}_{|R|-2}.$$
\end{lemma}

\begin{proof}[Proof of Lemma \ref{lem:LRH1}]
Consider a linear regression of $Y$ over $X$:
$$
Y=\beta_0+\beta X+\epsilon.
$$
Recall that
$$\widehat{\beta}=\frac{\sum_{i\in R}(X_{i}-\bar{X}_R )(Y_{i}-\bar{Y}_R )}{\sum_{i\in R}(X_{i}-\bar{X}_R )^{2}}$$
and $\widehat{\beta}_0=\bar{Y}_R -\widehat{\beta}\bar{X}_R $ are the least squares estimate of of $Y$ over $X$ where
$$
\bar{X}_R ={1\over |R|}\sum_{i\in R} X_i\qquad {\rm and}\qquad \bar{Y}_R ={1\over |R|}\sum_{i\in R} Y_i.
$$
The residual sum of squares of the regression can then be written as
$$
\sum_{i\in R}(Y_{i}-\widehat{\beta}_0-\widehat{\beta}X_{i})^{2}=\sum_{i\in R}(Y_{i}-\bar{Y}_R )^{2}(1-r_R^{2})
$$
Conditioned on $X_{i}$, the residual sum of squares will follow $(1-\rho^{2})\chi^{2}_{|R|-2}$. Thus the margin distribution of the residual sum of squares is also $(1-\rho^{2})\chi^{2}_{|R|-2}$.
\end{proof}
\vskip 25pt

Next we derive a tail bound for the sample correlation coefficient. For brevity, we work with the case when $(X,Y)$ are known to be centered so that
\begin{equation}
\label{eq:centerr}
r_R={\sum_{i\in R}X_iY_i\over \sqrt{\sum_{i\in R}X_i^2}\sqrt{\sum_{i\in R}Y_i^2}}
\end{equation}
where $(X_i,Y_i)$s are independent copies of $(X,Y)$. Treatment for the more general case is completely analogous, yet this simplification allows us to avoid lengthy discussions about the smaller order effects due to centering by sample means, and repeatedly switching between $|R|-1$ or $|R|-2$ as the appropriate degrees of freedom.

\begin{lemma}
\label{lem:r}
Assume that $\{(X_i,Y_i): i\in R\}$ are iid copies of $(X,Y)\sim N(0,I_2)$. Then for any $x>0$,
$$
\PP\left\{\left|\sum_{i\in R}X_{i}Y_{i}\right|\ge 2\sqrt{x|R|}+2x\right\}\le 4e^{-x}.
$$
If in addition, $0<x< |R|/16$, then
$$
\PP\{|r_R|\ge x\}\le 6\exp(-|R|x^2/64).
$$
\end{lemma}
\begin{proof}[Proof of Lemma \ref{lem:r}] Write
$$
\sum_{i\in R}X_{i}Y_{i}={1\over 2}\sum_{i\in R}\left({1\over \sqrt{2}}(X_i+Y_i)\right)^2-{1\over 2}\sum_{i\in R}\left({1\over \sqrt{2}}(X_i-Y_i)\right)^2.
$$
Then
\begin{eqnarray*}
\PP\left\{\left|\sum_{i\in R}X_{i}Y_{i}\right|\ge 2\sqrt{u|R|}+2u\right\}&\le& \PP\left\{\left|\sum_{i\in R}\left({1\over \sqrt{2}}(X_i+Y_i)\right)^2-|R|\right|\ge 2\sqrt{u|R|}+2u\right\}\\
&&+\PP\left\{\left|\sum_{i\in R}\left({1\over \sqrt{2}}(X_i-Y_i)\right)^2-|R|\right|\ge 2\sqrt{u|R|}+2u\right\}\\
&\le&4e^{-u},
\end{eqnarray*}
where the second inequality follows from the $\chi^2$ upper and lower tail bound of \cite{ChiSqu}. Applying the $\chi^2$ lower tail bound from \cite{ChiSqu}, we can also derive that
\begin{equation}
\label{eq:chi2x}
\PP\left\{\left|\sum_{i\in R}X_{i}^2\right|\le |R|-2\sqrt{u|R|}\right\}\le e^{-u}
\end{equation}
and
\begin{equation}
\label{eq:chi2y}
\PP\left\{\left|\sum_{i\in R}Y_{i}^2\right|\le |R|-2\sqrt{u|R|}\right\}\le e^{-u}
\end{equation}
Therefore, for any $u<|R|/16$,
$$
|r_{R}|\le {2\sqrt{u|R|}+2u\over |R|-2\sqrt{u|R|}}\le {2\over |R|}\left(2\sqrt{u|R|}+2u\right)\le 8\sqrt{u\over |R|}.
$$
with probability at least $1-6e^{-u}$. The claim follows immediately.
\end{proof}
\vskip 25pt

We are also interested in the difference in correlation coefficients between two different regions. The following lemma provides a useful probabilisitic tool for such purposes.
\begin{lemma}
\label{lem:difr}
Assume that $\{(X_i,Y_i): i\in R_1\cup R_2\}$ are iid copies of $(X,Y)\sim N(0,I_2)$, and $2|R_1\cap R_2|\ge |R_1\cup R_2|$. Then there exist numerical constants $c_0,c_1,c_2>0$ such that for any $x<c_0|R_1|$,
\begin{eqnarray*}
\PP(\left||R_1|r_{R_1}^2-|R_2|r_{R_2}^2\right|\ge x)\le \hskip 250pt\\
c_1\exp\left(-c_2\min\left\{\left({|R_1\cap R_2|\over |R_1\cup R_2|-|R_1\cap R_2|}\right)^{1/2}x,|R_1\cap R_2|^{1/3}x^{2/3}\right\}\right).
\end{eqnarray*}
In particular, if
$$
\zeta:={|R_1\cap R_2|\over \sqrt{|R_1||R_2|}}\ge {1\over 4},
$$
then there exists a numerical constant $c_3>0$ such that for any $x<c_0|R_1|$,
$$
\PP(\left||R_1|r_{R_1}^2-|R_2|r_{R_2}^2\right|\ge x)\le c_1\exp\left(-c_3\min\left\{ (1-\zeta)^{-1/2}x, |R_1\cap R_2|^{1/3}x^{2/3}\right\}\right).
$$
\end{lemma}

\begin{proof}[Proof of Lemma \ref{lem:difr}] We first consider the case when $R_2\subseteq R_1$. Recall that
$$
r_{R_1}=\frac{\sum_{i\in R_{1}}X_{i}Y_{i}}{\sqrt{\sum_{i\in R_{1}}X_{i}^{2}\sum_{i\in R_{1}}Y_{i}^{2}}}, \qquad {\rm and} \qquad
r_{R_2}=\frac{\sum_{i\in R_{1}}X_{i}Y_{i}}{\sqrt{\sum_{i\in R_{2}}X_{i}^{2}\sum_{i\in R_{2}}Y_{i}^{2}}}.
$$
Therefore,
\begin{eqnarray*}
|R_2|r_{R_2}^2-|R_1|r_{R_1}^2&=& \left({1\over \sqrt{|R_2|}}\sum_{i\in R_{2}}X_{i}Y_{i}\right)^2\left(\frac{|R_2|^2}{{\sum_{i\in R_{2}}X_{i}^{2}\sum_{i\in R_{2}}Y_{i}^{2}}}-\frac{|R_1|^2}{{\sum_{i\in R_{1}}X_{i}^{2}\sum_{i\in R_{1}}Y_{i}^{2}}}\right)\\
&&+\frac{|R_1|^2}{{\sum_{i\in R_{1}}X_{i}^{2}\sum_{i\in R_{1}}Y_{i}^{2}}}\left[{1\over |R_2|}\left(\sum_{i\in R_{2}}X_{i}Y_{i}\right)^2-{1\over |R_1|}\left(\sum_{i\in R_{1}}X_{i}Y_{i}\right)^2\right]
\end{eqnarray*}
We now bound the terms on the right hand side separately.

Observe that
\begin{eqnarray*}
&&\left|\left({1\over \sqrt{|R_2|}}\sum_{i\in R_{2}}X_{i}Y_{i}\right)^2\left(\frac{|R_1|^2}{{\sum_{i\in R_{1}}X_{i}^{2}\sum_{i\in R_{1}}Y_{i}^{2}}}-\frac{|R_2|^2}{{\sum_{i\in R_{2}}X_{i}^{2}\sum_{i\in R_{2}}Y_{i}^{2}}}\right)\right|\\
&=&|R_2|r_{R_2}^2\left|1-\left({|R_1|\over |R_2|}{\sum_{i\in R_{2}}X_{i}^{2}\over \sum_{i\in R_{1}}X_{i}^{2}}\right)^{-1}\left({|R_1|\over |R_2|}{\sum_{i\in R_{2}}Y_{i}^{2}\over \sum_{i\in R_{1}}Y_{i}^{2}}\right)^{-1}\right|.
\end{eqnarray*}
By Lemma \ref{lem:ftail},
$$
\PP\left\{\left|{|R_1|\over |R_2|}{\sum_{i\in R_{2}}X_{i}^{2}\over \sum_{i\in R_{1}}X_{i}^{2}}-1\right|\ge x\right\}\le 2\exp\left(-{|R_2|\over 12} x^2\right),
$$
and
$$
\PP\left\{\left|{|R_1|\over |R_2|}{\sum_{i\in R_{2}}Y_{i}^{2}\over \sum_{i\in R_{1}}Y_{i}^{2}}-1\right|\ge x\right\}\le 2\exp\left(-{|R_2|\over 12} x^2\right).
$$
We get, for any $x<1/2$,
$$
\left|1-\left({|R_1|\over |R_2|}{\sum_{i\in R_{2}}X_{i}^{2}\over \sum_{i\in R_{1}}X_{i}^{2}}\right)^{-1}\left({|R_1|\over |R_2|}{\sum_{i\in R_{2}}Y_{i}^{2}\over \sum_{i\in R_{1}}Y_{i}^{2}}\right)^{-1}\right|\le 4x
$$
with probability at least $1-4\exp\left(-|R_2|x^2/12\right)$. On the other hand, by Lemma \ref{lem:r},
$$
\PP\{|r_{R_2}|\ge x\}\le 6\exp(-|R_2|x^2/64).
$$
Thus, by taking $u=4|R_2|x^3$,
$$
|R_2|r_{R_2}^2\left|1-\left({|R_1|\over |R_2|}{\sum_{i\in R_{2}}X_{i}^{2}\over \sum_{i\in R_{1}}X_{i}^{2}}\right)^{-1}\left({|R_1|\over |R_2|}{\sum_{i\in R_{2}}Y_{i}^{2}\over \sum_{i\in R_{1}}Y_{i}^{2}}\right)^{-1}\right|\le u
$$
with probability at least
$$
1-10\exp\left(-{|R_2|^{1/3}u^{2/3}\over 64\cdot 4^{2/3}}\right).
$$
Denote by $\Ecal_1(u)$ the event that the above inequality holds.

To bound the second term, first note that
\begin{eqnarray*}
&&{1\over |R_2|}\left(\sum_{i\in R_{2}}X_{i}Y_{i}\right)^2-{1\over |R_1|}\left(\sum_{i\in R_{1}}X_{i}Y_{i}\right)^2\\
&=&\left({1\over |R_2|}-{1\over |R_1|}\right)\left(\sum_{i\in R_{1}}X_{i}Y_{i}\right)^2-{1\over |R_2|}\left(\sum_{i\in R_{1}\setminus R_2}X_{i}Y_{i}\right)^2-{2\over |R_2|}\left(\sum_{i\in R_2}X_{i}Y_{i}\right)\left(\sum_{i\in R_{1}\setminus R_2}X_{i}Y_{i}\right).
\end{eqnarray*}
By Lemma \ref{lem:r},
$$
\frac{|R_1|^2}{{\sum_{i\in R_{1}}X_{i}^{2}\sum_{i\in R_{1}}Y_{i}^{2}}}\left({1\over |R_2|}-{1\over |R_1|}\right)\left(\sum_{i\in R_{1}}X_{i}Y_{i}\right)^2=\left({|R_1|\over |R_2|}-1\right)|R_1|r_{R_1}^2\le \left({|R_1|\over |R_2|}-1\right)|R_1|x^2,
$$
with probability at least $1-6\exp(-|R_1|x^2/64)$. On the other hand, again by Lemma \ref{lem:r},
$$
\PP\left\{\left|\sum_{i\in R_{1}\setminus R_2}X_{i}Y_{i}\right|\ge  2\sqrt{x(|R_1|-|R_2|)}+2x\right\}\le 4e^{-x},
$$
Recall that, by $\chi^2$ lower tail bounds from \cite{ChiSqu}, we get
$$
\PP\left\{\sum_{i\in R_1}X_{i}^2\le |R_1|-2\sqrt{x|R_1|}\right\}\le e^{-x},
$$
and
$$
\PP\left\{\sum_{i\in R_1}Y_{i}^2\le |R_1|-2\sqrt{x|R_1|}\right\}\le e^{-x}.
$$
Thus, for any $x<|R_1|/16$,
\begin{eqnarray*}
\frac{|R_1|^2/|R_2|}{{\sum_{i\in R_{1}}X_{i}^{2}\sum_{i\in R_{1}}Y_{i}^{2}}}\left(\sum_{i\in R_{1}\setminus R_2}X_{i}Y_{i}\right)^2&\le& {|R_1|^2/|R_2|\over \left(|R_1|-2\sqrt{x|R_1|}\right)^2}\left(2\sqrt{x(|R_1|-|R_2|)}+2x\right)^2\\
&\le& {16\over |R_2|}\left(\sqrt{x(|R_1|-|R_2|)}+x\right)^2,
\end{eqnarray*}
with probability at least $1-6e^{-x}$. In other words,
\begin{eqnarray*}
\frac{|R_1|^2/|R_2|}{{\sum_{i\in R_{1}}X_{i}^{2}\sum_{i\in R_{1}}Y_{i}^{2}}}\left(\sum_{i\in R_{1}\setminus R_2}X_{i}Y_{i}\right)^2&\le& {1\over |R_2|}\left({1\over 2}x\sqrt{|R_1|(|R_1|-|R_2|)}+{|R_1|x^2\over 16}\right)^2\\
&\le&{1\over 2}\left({|R_1|\over |R_2|}-1\right)|R_1|x^2+{|R_1|^2x^4\over 128|R_2|},
\end{eqnarray*}
with probability at least $1-6\exp(-|R_1|x^2/64)$ for any $x<2$. Following a similar argument, we can also show that
\begin{eqnarray*}
&&\frac{2|R_1|^2/|R_2|}{{\sum_{i\in R_{1}}X_{i}^{2}\sum_{i\in R_{1}}Y_{i}^{2}}}\left(\sum_{i\in R_2}X_{i}Y_{i}\right)\left(\sum_{i\in R_{1}\setminus R_2}X_{i}Y_{i}\right)\\
&\le& {1\over |R_2|}\left({1\over 2}x\sqrt{|R_1||R_2|}+{|R_1|x^2\over 16}\right)\left({1\over 2}x\sqrt{|R_1|(|R_1|-|R_2|)}+{|R_1|x^2\over 16}\right),
\end{eqnarray*}
with probability at least $1-10\exp(-|R_1|x^2/64)$ for any $x<2$. Note $2|R_2|\ge |R_1|$. In summary, we get
\begin{eqnarray*}
&&\frac{|R_1|^2}{{\sum_{i\in R_{1}}X_{i}^{2}\sum_{i\in R_{1}}Y_{i}^{2}}}\left|{1\over |R_2|}\left(\sum_{i\in R_{2}}X_{i}Y_{i}\right)^2-{1\over |R_1|}\left(\sum_{i\in R_{1}}X_{i}Y_{i}\right)^2\right|\\
&\le&2\left({|R_1|\over |R_2|}-1\right)^{1/2}|R_1|x^2+{|R_1|x^3\over 16}
\end{eqnarray*}
with probability at least $1-22\exp(-|R_1|x^2/64)$ for any $x<2$. Hence, with probability at least
$$
1-22\exp\left(-{1\over 256}\left({|R_1|\over |R_2|}-1\right)^{-1/2}u\right)-22\exp\left(-{1\over 16}|R_1|^{1/3}u^{2/3}\right)
$$
we have
$$
\frac{|R_1|^2}{{\sum_{i\in R_{1}}X_{i}^{2}\sum_{i\in R_{1}}Y_{i}^{2}}}\left|{1\over |R_2|}\left(\sum_{i\in R_{2}}X_{i}Y_{i}\right)^2-{1\over |R_1|}\left(\sum_{i\in R_{1}}X_{i}Y_{i}\right)^2\right|\le u.
$$
Denote this event by $\Ecal_2(u)$.

In summary, for any $u<|R_1|/256$,
$$
||R_1|r_{R_1}^2-|R_2|r_{R_2}^2|\le 2u
$$
with probability at least
\begin{eqnarray*}
\PP\left\{\Ecal_1(u)\bigcap\Ecal_2(u)\right\}&\ge& 1-22\exp\left(-{1\over 256}\left({|R_1|\over |R_2|}-1\right)^{-1/2}u\right)-22\exp\left(-{1\over 16}|R_1|^{1/3}u^{2/3}\right)\\
&&\qquad -10\exp\left(-{|R_2|^{1/3}u^{2/3}\over 64\cdot 4^{2/3}}\right)\\
&\ge&1-22\exp\left(-{1\over 256}\left({|R_1|\over |R_2|}-1\right)^{-1/2}u\right)-32\exp\left(-{1\over 128}|R_2|^{1/3}u^{2/3}\right).
\end{eqnarray*}
The statement, when $R_2\subseteq R_1$, then follows.

Now consider the general case when $R_2\nsubseteq R_1$. In this case,
$$
||R_1|r_{R_1}^2-|R_2|r_{R_2}^2|\le ||R_1|r_{R_1}^2-|R_1\cap R_2|r_{R_1\cap R_2}^2|+||R_2|r_{R_2}^2-|R_1\cap R_2|r_{R_1\cap R_2}^2|.
$$
We can now appeal to the bounds we derived for nested sets before to get
\begin{eqnarray*}
\PP(||R_1|r_{R_1}^2-|R_1\cap R_2|r_{R_1\cap R_2}^2|\ge x)\le 22\exp\left(-{1\over 256}\left({|R_1\cap R_2|\over |R_1|-|R_1\cap R_2|}\right)^{1/2}u\right)\\
+32\exp\left(-{1\over 128}|R_1\cap R_2|^{1/3}u^{2/3}\right),
\end{eqnarray*}
and
\begin{eqnarray*}
\PP(||R_2|r_{R_2}^2-|R_1\cap R_2|r_{R_1\cap R_2}^2|\ge x)\le 22\exp\left(-{1\over 256}\left({|R_1\cap R_2|\over |R_2|-|R_1\cap R_2|}\right)^{1/2}u\right)\\
+32\exp\left(-{1\over 128}|R_1\cap R_2|^{1/3}u^{2/3}\right).
\end{eqnarray*}
The first claim then follows from an application of the union bound.

To show the second statement, assume that $|R_1|\ge |R_2|$ without loss of generality. Observe that
$$
\rho={|R_1\cap R_2|\over \sqrt{|R_1||R_2|}}\le \sqrt{|R_2|\over |R_1|},
$$
which implies that $|R_2|\ge \rho^2|R_1|$. Therefore,
$$
|R_1\cap R_2|=\rho\sqrt{|R_1||R_2|}\ge \rho^2|R_1|,
$$
and
$$
|R_1\cap R_2|=\rho\sqrt{|R_1||R_2|}\ge \rho|R_2|.
$$
Thus,
$$
{|R_1\cap R_2|\over |R_1\cup R_2|-|R_1\cap R_2|}\ge {1\over \rho^{-2}+\rho^{-1}-2}={1\over 1-\rho}{\rho^2\over 2\rho+1}\ge {1\over 48}(1-\rho)^{-1},
$$
where the last inequality follows from the fact that $1/4\le \rho\le 1$.
\end{proof}
\vskip 25pt

We are now in position to derive bounds for the likelihood ratio statistic $L_R$. Since we work with centered random variables as stated earlier, it is natural to redefine $L_R$ as:
$$
L_R=-(|R|-1)\log(1-r_{R}^{2}).
$$
where $r_R$ is given by (\ref{eq:centerr}).

\begin{lemma}
\label{lem:LR}
Assume that $\{(X_i,Y_i): i\in R\}$ are iid copies of $(X,Y)\sim N(0, I_2)$ for some $|R|>1$. Then there exists numerical constants $0<c_1<c_2$ such that for any $x>1$,
$$
c_1|R|^{-1/2}e^{-x/2}\le\PP(L_R>x)\le c_2e^{-x/2}.
$$
\end{lemma}

\begin{proof}[Proof of Lemma \ref{lem:LR}]
Observe that
$$
L_{R}=(|R|-1)\log{\left(1+\frac{T_{R}^{2}}{|R|-1}\right)}
$$
where
$$
T_R=r_{R}\sqrt{\frac{|R|-1}{1-r_{R}^{2}}}
$$
and
$$
r_{R}=\frac{\sum_{i\in R}X_{i}Y_{i}}{\sqrt{\sum_{i\in R}X_{i}^{2}\sum_{i\in R}Y_{i}^{2}}}
$$
It is well known that, under the null hypothesis,
$$
T_R\sim t_{|R|-1}.
$$
See, e.g., \cite{RhotDis}. By Lemma \ref{lem:tconcen},
$$
c_1|R|^{-1/2}e^{-x/2}\le \PP(L_R>x)\le c_2e^{-x/2},
$$
for any $x>1$.
\end{proof}
\vskip 25pt

The following lemma bounds the change in the likelihood ration statistic due to a perturbation of the index set.

\begin{lemma}
\label{lem:dif}
Assume that $\{(X_i,Y_i): i\in R_1\cup R_2\}$ are iid copies of $(X,Y)\sim N(0,I_2)$, and $2|R_1\cap R_2|\ge |R_1\cup R_2|$. Then there exist numerical constants $c_0,c_1,c_2>0$ such that for any $x<c_0|R_1|$,
$$
\PP(\left|L_{R_1}-L_{R_2}\right|\ge x)\le c_1\exp\left(-c_2\min\left\{\left({|R_1\cap R_2|\over |R_1\cup R_2|-|R_1\cap R_2|}\right)^{1/2}x,|R_1\cap R_2|^{1/3}x^{2/3}\right\}\right).
$$
In particular, if
$$
\zeta:={|R_1\cap R_2|\over \sqrt{|R_1||R_2|}}\ge {1\over 4},
$$
then there exists a numerical constant $c_3>0$ such that for any $x<c_0|R_1|$,
$$
\PP(\left|L_{R_1}-L_{R_2}\right|\ge x)\le c_1\exp\left(-c_3\min\left\{ (1-\zeta)^{-1/2}x, |R_1\cap R_2|^{1/3}x^{2/3}\right\}\right).
$$
\end{lemma}

\begin{proof}[Proof of Lemma \ref{lem:dif}] Similar to Lemma \ref{lem:difr}, it suffices to prove the first statement when $R_2\subseteq R_1$. By the convexity of $-\log(1-x)$, we can ensure
\[L_{R_{1}}=-|R_{1}|\log(1-r_{R_{1}}^{2})\ge -|R_{1}|\log(1-r_{R_{2}}^{2})+\frac{|R_{1}|(r_{R_{1}}^{2}-r_{R_{2}}^{2})}{1-r_{R_{2}}^{2}}\]
and
\[L_{R_{2}}=-|R_{2}|\log(1-r_{R_{2}}^{2})\ge -|R_{2}|\log(1-r_{R_{1}}^{2})+\frac{|R_{2}|(r_{R_{2}}^{2}-r_{R_{1}}^{2})}{1-r_{R_{1}}^{2}}\]
Therefore,
\begin{equation}
\label{eq:Ldiff}
|L_{R_{1}}-L_{R_{2}}|\le  (|R_2|-|R_1|)\log(1-\max\{r_{R_{1}}^{2},r_{R_{2}}^{2}\})+\frac{|R_{2}||r_{R_{2}}^{2}-r_{R_{1}}^{2}|}{1-\max\{r_{R_{1}}^{2},r_{R_{2}}^{2}\}}.
\end{equation}
We now bound the two terms on the right hand side separately.

Denote by $\Ecal(\alpha)$ the event that
$$\max\{r_{R_{1}}^{2},r_{R_{2}}^{2}\}< \alpha.$$
By Lemma \ref{lem:r},
$$
\PP\{\Ecal(\alpha)\}\ge 1-12\exp(-|R_2|\alpha/64).
$$

Note that, for any $0<x<\alpha$,
$$
-\log(1-x)\le {x\over 1-\alpha}.
$$
We can upper bound the first term on the right hand side of (\ref{eq:Ldiff}) by
$$
{1\over 1-\alpha}\left(|R_1|-|R_2|\right)\max\{r_{R_{1}}^{2},r_{R_{2}}^{2}\}.
$$
Therefore,
\begin{equation}
\label{eq:Lbd1}
\PP\{(|R_2|-|R_1|)\log(1-\max\{r_{R_{1}}^{2},r_{R_{2}}^{2}\})\ge u\}\le 12\exp\left(-{1\over 64}|R_2|\min\left\{\alpha, {(1-\alpha)u\over |R_1|-|R_2|}\right\}\right).
\end{equation}

The second term of (\ref{eq:Ldiff}) can be upper bounded by
$$
{1\over 1-\alpha}\left(||R_2|r_{R_2}^2-|R_1|r_{R_1}^2|+(|R_1|-|R_2|)r_{R_1}^2\right),
$$
under the event $\Ecal(\alpha)$. By Lemma \ref{lem:difr}, we get
$$
\PP\{||R_2|r_{R_2}^2-|R_1|r_{R_1}^2|\ge x\}\le c_1\exp\left(-c_2\min\left\{\left({|R_2|\over |R_1|-|R_2|}\right)^{1/2}x,|R_2|^{1/3}x^{2/3}\right\}\right).
$$
And by Lemma \ref{lem:r},
$$
\PP\{(|R_1|-|R_2|)r_{R_1}^2\ge x\}\le 6\exp\left(-{|R_1|x\over 64(|R_1|-|R_2|)}\right)
$$
Therefore,
\begin{eqnarray*}
\PP\left\{\frac{|R_{2}||r_{R_{2}}^{2}-r_{R_{1}}^{2}|}{1-\max\{r_{R_{1}}^{2},r_{R_{2}}^{2}\}}\ge u\right\}\le 1-12\exp(-|R_2|\alpha/64)-6\exp\left(-{(1-\alpha)|R_1|u\over 128(|R_1|-|R_2|)}\right)\\
-c_1\exp\left(-{1-\alpha\over 2}c_2\min\left\{\left({|R_2|\over |R_1|-|R_2|}\right)^{1/2}u,|R_2|^{1/3}u^{2/3}\right\}\right).
\end{eqnarray*}
Together with (\ref{eq:Lbd1}), this implies the desired statement for $R_2\subseteq R_1$.
\end{proof}
\vskip 25pt

A careful inspection of the derivation of Lemma \ref{lem:dif} suggests that it can be extended to a more general situation where $X$ and $Y$ are correlated for some indices.
\begin{lemma}
\label{lem:dif1}
Let $R_1\subset R_2$ be two index sets. Assume that $\{(X_i,Y_i): i\in R_1\}$ are independent observations so that $(X_i,Y_i)\sim N(0,I_2)$ for $i\in R_1$, and $X_i$, $Y_i$ are standard normal random variables with correlation coefficient $\rho$ for $i\notin R_1$. Then there exist numerical constants $c_0,c_1,c_2>0$ such that for any $x<c_0|R_1|$,
$$
\PP(\left|L_{R_1}-L_{R_2}\right|\ge x)\le c_1\exp\left(-c_2\min\left\{ (1-\zeta)^{-1/2}x, |R_1|^{1/3}x^{2/3}\right\}\right).
$$
provided that $\zeta:={|R_1|/|R_2|}\ge 1/4$.
\end{lemma}

Finally we derive a perturbation bounds for a polygon which is useful for our discussion in Section 3.1.

\begin{lemma}
\label{lem:setdif}
Let $K_1$ and $K_2$ be two polygons with vertices $u_{1}, u_{2},\ldots, u_{k}$ and $v_{1}, v_{2},\ldots, v_{k}$ respectively. Denote by $e_{j}$ the length of the edge between $u_j$ and $u_{(j {\rm mod} k)+1}$. Then
$$
|K_{1}\cap K_{2}^{c}|\le r\sum_{j=1}^{k}(e_{j}+2r),
$$
where $r$ is maximum distance between $u_j$ and $v_j$.
\end{lemma}

\begin{proof}
Denote by $Q_{i}$ the polygon whose first $i$ vertices are the same with $K_{2}$ and whose remaining vertices are the same with $K_{1}$. In particular, $Q_{0}=K_{1}$ and $Q_{k}=K_{2}$. It is not hard to see that the $j$th edge of $Q_{i}$ is no longer than $e_{j}+2r$. If we compare $Q_{0}$ and $Q_{1}$, then only the first vertex might be different, as illustrated in Figure \ref{fig:perturb}.

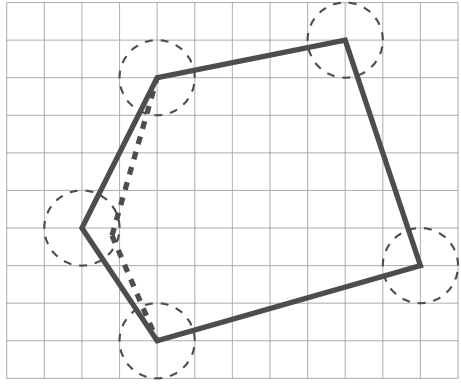
\begin{figure}[htbp]
\begin{center}
\begin{tikzpicture}
 \begin{scope}
       \fill[white,fill opacity=0.8] (0,-1) rectangle (6,4);
       \draw[step=0.5, black!30!white] (0,-1) grid (6,4);
       \path[draw,line width=2pt,black!70!white] (1,1) -- (2,3)--(4.5,3.5)--(5.5,0.5)--(2,-0.5)--cycle;
       \draw[dashed,thick,black!70!white] (1,1) circle (0.5);
       \path[draw,line width=2pt,black!70!white,dashed] (1.4,0.9) -- (2,3);
       \path[draw,line width=2pt,black!70!white,dashed] (1.4,0.9) -- (2,-0.5);
       \draw[dashed,thick,black!70!white] (4.5,3.5) circle (0.5);
       \draw[dashed,thick,black!70!white] (5.5,0.5) circle (0.5);
       \draw[dashed,thick,black!70!white] (2,3) circle (0.5);
       \draw[dashed,thick,black!70!white] (2,-0.5) circle (0.5);
 \end{scope}
\end{tikzpicture}
\end{center}
\caption{Effect of perturbation of vertices of a polygon.}
\label{fig:perturb}
\end{figure}

Because $Q_0$ and $Q_1$ are different only in the first vertex, they can only be different in the two edges linked with the first index. It can then be computed that
$$
|Q_{0}\cap Q_{1}^{c}|\le \frac{1}{2}r(e_{1}+e_{k}+4r)
$$
Similarly,
$$
|Q_{i}\cap Q_{i+1}^{c}|\le \frac{1}{2}r(e_{i}+e_{i+1}+4r),\qquad  i=1,2,\ldots,k-1,
$$
It is clear that
$$
K_{1}\cap K_{2}^{c}=Q_{0}\cap Q_{k}^{c}\subset \cup_{i=0}^{k-1}(Q_{i}\cap Q_{i+1}^{c})
$$
Therefore,
$$
|K_{1}\cap K_{2}^{c}|\le r\sum_{i=1}^{k}(e_{i}+2r),
$$
which completes the proof.
\end{proof}
\vskip 25pt

\begin{proof}[Proof of Proposition \ref{pr:poly1}]
Write
$$
\Ccal_{p_{1},\ldots,p_{k}}=\{K(\{(a_i,b_i): 1\le i\le k\}):2^{p_{i}}\le r_{i}< 2^{p_{i}+1},i=1,\ldots,k\}.
$$
It is clear that there exists a constant $C>0$ such that
$$
|\Ccal_{p_{1},\ldots,p_{k}}|\le C n2^{2\left(\sum_{i=1}^{k}p_{i}\right)}.
$$
Note that there are constants $c_1,c_2>0$ depending on $k$ and $M$ only such that
$$
\Rcal_{\rm polygon}(A;k,M)\subset \left\{K\in\Rcal_{\rm polygon}(k,M): c_1A^{1/2}\le r_i\le c_2A^{1/2}, i=1,2,\ldots,k\right\}.
$$
Therefore,
$$
\Rcal_{\rm polygon}(A;k,M)\le cnA^{k}
$$
which completes the proof because $A\asymp r_i^2$.
\end{proof}
\vskip 25pt

\begin{proof}[Proof of Proposition \ref{pr:poly2}]
Note that $\pi_s(K(\{(a_i,b_i): 1\le i\le k\}))$ is also a polygon. For brevity, we shall hereafter denote it by $K(\{(\tilde{a}_i,\tilde{b}_i): 1\le i\le k\})$. By Lemma \ref{lem:setdif}, we get
$$
|K(\{(a_i,b_i): 1\le i\le k\})\setminus K(\{(\tilde{a}_i,\tilde{b}_i): 1\le i\le k\})|\le C2^{s}\sum_{i}r_i\le   Ck2^sr_1.
$$
Hence
$$
\rho\left(K(\{(a_i,b_i): 1\le i\le k\}),K(\{(\tilde{a}_i,\tilde{b}_i): 1\le i\le k\})\right)\ge 1-\frac{Ck2^sr_1}{\pi r_1^2}\ge 1-\frac{C2^{s}}{r_1},
$$
which completes the proof.
\end{proof}

\end{document}